\documentclass[11pt,oneside,english]{amsart}
\usepackage{amssymb}

\makeatletter \theoremstyle{plain}
 \newtheorem{thm}{Theorem}[section]
 \newtheorem{lem}[thm]{Lemma}
 \newtheorem{prop}[thm]{Proposition}
 \newtheorem{corollary}[thm]{Corollary}
 \numberwithin{equation}{section} %% Comment out for sequentially-numbered
 \numberwithin{figure}{section} %% Comment out for sequentially-numbered
 \theoremstyle{plain}
 \theoremstyle{definition}
 \newtheorem{defn}[thm]{Definition}
 
\newcommand{\A}{{{\mathcal A}}}

\newcommand{\calS}{{{\mathcal S}}}
\newcommand{\calC}{{{\mathcal C}}}
\newcommand{\fH}{{{\mathfrak H}}}
\newcommand{\fM}{{{\mathfrak M}}}

\newcommand{\calL}{{{\mathcal L}}}
\newcommand{\E}{{{\mathcal S}}}

\newcommand{\C}{{{\mathbb C}}}

\newcommand{\R}{{{\mathbb R}}}

\newcommand{\J}{{{\mathbb J}}}

\usepackage{babel}
\makeatother

\begin{document}

\title [Modulus  of surface families - Radial stretch]{Modulus of surface families  and the radial stretch in the heisenberg group}

\author{Ioannis D. Platis}

\address {Department of Mathematics and Applied Mathematics,
University of Crete, 
University Campus,
GR 70013 Heraklion Crete,
 Greece.}

\email{jplatis@math.uoc.gr}

\begin{abstract}
We develop a modulus method for surface families inside a domain in the Heisenberg group and we prove that  the stretch map between two Heisenberg spherical rings   is a minimiser for the mean distortion  among the class of contact quasiconformal maps between these rings which satisfy certain boundary conditions.

\end{abstract}

\maketitle

\section{Introduction and Statement of Results}

In the classical theory of quasiconformal mappings of the complex plane $\C$, various tools have been developed for the solution of the so called extremal problems: given a family $\mathcal{F}$ of quasiconformal  mappings $f$ between two given regions $\Omega$ and $\Omega'$ of $\C$ subject to boundary conditions, find a minimiser of the {\it maximal distortion} 
$$
K_f=\sup_{p\in \Omega}K_f(p),\quad K_f(p)=\frac{|f_z(p)|+|f_{\overline{z}}(p)|}{|f_z(p)|-|f_{\overline{z}}(p)|},\quad f\in\mathcal{F},
$$ 
or a minimiser of the {\it mean distortion} 
$$
\fM(f,\rho_0)=\frac{\iint_\Omega K_f(p)\rho_0^2(p)d\calL^2(p)}{\iint_\Omega \rho_0^2(p)d\calL^2(p)},\quad f\in\mathcal{F},
$$
where $\rho_0:\Omega\to\R^+$ is a positive Borel function related to the geometry of $\Omega$. 
 That is, in the first case, we wish to find an $f_0\in\mathcal{F}$ so that $K_{f_0}\le K_f$ for every $f\in\mathcal{F}$, whereas in the second case our minimiser has to satisfy $\fM(f_0,\rho_0)\le\fM(f,\rho_0)$  for every $f\in\mathcal{F}$.  Motivation to solve problems like these comes from various sources, we mention for instance theory of elasticity. The study of the first problem goes back to Gr\"otzsch, \cite{Gr},  Teichm\"uller, \cite{T},  Strebel, \cite{S}, to mention only a few. Modulus methods for curve families  and quadratic differentials theory constitute classical tools for the solution of this problem. The second problem is more difficult; as it was shown in \cite{Ma} there are examples where  a minimiser for the mean distortion does not exist. Nevertheless, several  other methods have been deployed for the study of this problem, see for instance \cite{AIMO}, \cite{BFP2}, \cite{GM}, \cite{Ma}. For an overview of the section of Teichm\"uller theory concerning extremal problems, in particular 
quadratic differentials theory, we refer the reader to the books of Strebel \cite{S1} and Gardiner and Lakic, \cite{Ga}; for modulus methods  we refer to the book of Vasil'ev, \cite{Va}.

In contrast to the classical theory where a rather rich machinery is available for the solution of minimisation problems of the aforementioned nature, in the Heisenberg group (or, more generally, in the sub-Riemannian) case the picture is not analogous. However, the combination of the powerful Kor\'anyi-Reimann theory of quasiconformal mappings of the Heisenberg group with moduli methods seem to yield interesting results. In this paper is to propose a method of modulus of surface families inside the Heisenberg group in order to minimise a mean distortion functional; we describe this method below.   

For the moment, we recall in brief a few facts about the Heisenberg group $\fH$; for details, see Section \ref{sec:prel}.  $\fH$ is $\C\times\R$ with group law
$$
(z,t)*(w,s)=(z+w,t+s+2\Im(\overline{w}z))
$$
and it admits a natural left invariant metric (the {\it Heisenberg distance}) defined by
\begin{equation*}
d_\fH(p,q)=\|p^{-1}*q\|,\quad p=(z,t),q=(w,s)\in\fH,
\end{equation*}
where $\|(z,t)\|=\left||z|^2-it\right|^{1/2}$.
The general metric definition of a quasiconformal mapping due to Mostow, \cite{M},  applies for  the Heisenberg group $\fH$ when it is  considered as a metric space with metric $d_\fH$.

\medskip

\begin{defn}\label{defn:metric}{\bf (Metric definition)}
A homeomorphism $f:\Omega\to\Omega'$ between domains $\Omega$ and $\Omega'$ in $\fH$ is quasiconformal if
$$
{\rm ess\; sup}_{p\in\Omega}H(p)<\infty,
$$
where
$$
H(p)=\limsup_{r\to 0}\frac{\max_{d_\fH(p,q)=r}d_\fH(f(p),f(q))}{\min_{d_\fH(p,q)=r}d_\fH(f(p),f(q))}.
$$
It is called $K$-{\it quasiconformal} if there is a constant $K\ge 1$ such that
$$
{\rm ess\;sup}_{p\in\Omega}H(p)\le K.
$$
\end{defn}

\medskip

In analogy to the complex case, there are equivalent analytic and geometric definitions, see Section \ref{sec:quas} for details. But in the Heisenberg group case there exists an additional condition for  $K$-quasiconformal mappings $f$. Namely, such an $f$ has to preserve the contact form $$\omega=dt+2\Im(\overline{z}dz)$$ of $\fH$, i.e., $f^*\omega=\lambda\omega$ a.e., for some non-vanishing function $\lambda$. Moreover, if $f=(f1,f_2,f_3)$ then the distributional derivatives $Zf_I$ and ${Z}f_I$
exist a.e. Where, $f_I=f_1+if_2$ and $Z,\;\overline{Z}$ are the {\it horizontal} vector fields
$$
Z=\frac{\partial}{\partial z}+i\overline{z}\frac{\partial}{\partial t},\quad Z=\frac{\partial}{\partial \overline{z}}-iz\frac{\partial}{\partial t}.
$$
If $f$ is orientation-preserving, then it satisfies a.e. the following Beltrami system of equations:
$$
\overline{Z}f_I=\mu_f Zf_I \quad\text{and}\quad \overline{Z}f_{II}=\mu_f Zf_{II}.
$$
Here, $f_{II}=f_3+i|f_I|^2$ and the Beltrami coefficient $\mu_f$  of $f$ is a measurable complex function which is  essentially bounded by a constant $k\in[0,1)$, i.e., $\|\mu_f\|_\infty=k$. For an {\it arbitrary} orientation-preserving homeomorphism $f$ such that $Zf_I$ and $\overline{Z}f_I$ exist a.e., we consider the Beltrami coefficient, the distortion function of $f$ and the the maximal distortion of $f$ which are given respectively by
$$
\mu_f(z,t)=\frac{\overline{Z}f_I(z,t)}{Zf_I(z,t)},\quad K_f(z,t)=\frac{|\mu_f(z,t)|+1}{|\mu_f(z,t)|-1},\quad K_f={\rm ess}\sup_{(z,t)}K_f(z,t).
$$
%and also the maximal distortion $K_f$ of $f$ which is the essential supremum of $K_f(z,t)$.
For a quasiconformal $f$ we have $\|\mu_f\|_\infty=k<1$ and thus $1\le K_f=\frac{1+k}{1-k}<\infty$. In this paper we consider $\calC^2$ orientation-preserving contact transformations such that $1\le K_f<\infty$. These have to be quasiconformal, see Section \ref{sec:quas}.
 
\medskip

As in the case of the complex plane, there are two kinds of extremal problems we are interested in: first, we seek for minimisers for the maximal distortion $K_f$ %as well as for minimisers of the mean distortion integral
 among a class $\mathcal{F}$ of sense preserving quasiconformal maps $f:\Omega\to\Omega'$ between domains of $\fH$ which satisfy certain boundary conditions. Standard family arguments and appropriate conditions on $\mathcal{F}$ ensure the existence of a minimiser for the maximal distortion, see Theorem F in \cite{KR1}; however, there is no standard method to detect such a minimiser. %For the first problem, and given the absence of a theory analogous to the one of quadratic differentials in the complex case, there is no standard method yet which detects such minimisers. Nevertheless, in some particular cases namely in cases where the distortion function $K_f(p)$ of a candidate minimiser $f$ is constant,  modulus methods may prove that $f$ is indeed a minimiser. But, up to now only a few   
For the second case,  it turns out that there are two kinds of mean distortions that arise naturally in the Heisenberg group case. First, we have  the {\it 2-mean distortion} $\fM_2(f,\rho)$ of a quasiconformal map $f$ which is given by
\begin{equation*}
\fM_2(f,\rho)=\frac{\iiint_\Omega K_f^2(p)\rho^4(p)d\calL^3(p)}{\iiint_\Omega \rho^4(p)d\calL^3(p)},
\end{equation*}
where $\rho:\Omega\to\R^+$ is a Borel function related to the geometry of $\Omega$. %This mean distortion integral appears in the Modulus Inequality for curve families, see \cite{BFP1}.
In \cite{BFP1}, Balogh, F\"assler and the author developed a method relying on  moduli of curve families inside a region of the Heisenberg group, resulting to the following theorem. %that a particular mapping between two Heisenberg spherical rings $S_{a,b}$ and $S_{a^k,b^k}$ which we call the {\it stretch map} and we denote it by $f_k$, is in fact a minimiser for the 2--mean distortion within a family of mappings between these two regions subject to certain boundary conditions. 

\medskip

\begin{thm}\label{thm:shrink}
For any $k\in(0,1)$, the stretch map $f_k$ is a sense preserving quasiconformal map
from the Heisenberg spherical ring $S_{a,b}$ onto the Heisenberg spherical ring
$S_{a^k,b^k}$ with maximal distortion $K_{f_k}=k^{-2}$.  Denote by $\mathcal{F}$ the class of quasiconformal maps $f:S_{a,b}\to S_{a^k,b^k}$
which preserve the axis $(0,t)$ and map the boundary components of $S_{a,b}$ to the respective boundary
components of $S_{a^k,b^k}$. Then up to composition with rotations around the vertical axis the stretch map $f_k$ minimises the 2-mean
distortion within the class $\mathcal{F}$: for any $f\in\mathcal{F}$ we have
$$
k^{-3}=K_{f_k}^{3/2}=\fM_2(f_k,\rho_0)\le\fM_2(f,\rho_0),
$$
%$$
%k^{-3}=K_{f_k}^{3/2}=\frac{\iiint_{S_{a,b}}K_{f_k}^2(p)\rho_0^4(p) d\calL^3(p)}{\iiint_{S_{a,b}}\rho_0^4(p) d\calL^3(p)}\le
%\frac{\iiint_{S_{a,b}}K_f^2(p)\rho_0^4(p) d\calL^3(p)}{\iiint_{S_{a,b}}\rho_0^4(p) d\calL^3(p)},
%$$
where
$$
\rho_0(z,t)=\frac{1}{\log(b/a)}\cdot \frac{|z|}{\left||z|^2-it\right|}\cdot \mathcal{X}
(S_{a,b})(z,t),\quad\text{for every}\quad (z,t)\in\fH.
$$
Moreover,
$
K_f\ge K_{f_k}^{3/4}.
$
\end{thm}

\medskip

The stretch map $f_k$ of Theorem \ref{thm:shrink} (which is actually a {\it shrink} map since $k\in(0,1)$) is a generalisation of the classical stretch map
$$
g_k(z)=z|z|^{k-1},\quad z\in\C.
$$
In cartesian coordinates $f_k$ is given by the  formula
$$
f_k(z,t)=\left(k^{1/2}z\left(\frac{\overline{w}}{k|z|^2+it}\right)^{1/2}|w|^{\frac{k-1}{2}},\;t\cdot\frac{|w|^k}{k|z|^2+it}\right),
$$
where $w=|z|^2-it$. Observe that if we set $t=0$ we find $f_k(z,0)=(g_k(z),0)$ and in this way we recover the classical stretch. Nevertheless, the cartesian expression of $f_k$ is rather complicated;  the stretch  is best described in the parametrisation of $\fH$ by {\it logarithmic coordinates} $(\xi,\psi,\eta)$. This parametrisation is given by 
$$
(\xi,\psi,\eta)\mapsto\left(i\cos^{1/2}\psi e^{\frac{\xi+i(\psi-3\eta)}{2}},-\sin\psi e^{\xi}\right),
$$
where $\xi\in\R$, $\psi\in[-\pi/2,\pi/2]$ and $\psi-3\pi\le 3\eta\le\psi=\pi$ (for details, see Section \ref{sec:log} as well as \cite{BFP1} and \cite{P}). Those coordinates are  analogous to the logarithmic coordinates of the complex plane; the contact form, the contact conditions and the Beltrami equations admit expressions which are considerably easier to handle than in the cartesian case. In fact, the expression for the stretch  $f_k$ is 
$$
f_k(\xi,\psi,\eta)=\left(k\xi,\arctan\left(\frac{\tan\psi}{k}\right),\eta\right).
$$
%Now, the main result of \cite{BFP1} is the following.
%A significant difference between the classical stretch and the Heisenberg stretch occurs immediately: the Heisenberg stretch is not a minimiser for the maximal distortion whereas the classical stretch is.
Note finally that the above result holds only for $k\in(0,1)$; for $k>1$ the particular modulus method does not work. % We describe in brief the modulus method which used to deduce Theorem \ref{thm:shrink}. 

We now define the {\it 2/3-mean distortion } $\fM_{2/3}(f,\rho_0)$ of a quasiconformal transformation $f$  by
\begin{equation*}
\fM_{2/3}(f,\rho_0)=\frac{\iiint_\Omega K_f^{2/3}(p)\rho_0^{4/3}(p)d\calL^3(p)}{\iiint_\Omega \rho^{4/3}(p)d\calL^3(p)},
\end{equation*}
where again  $\rho_0:\Omega\to\R^+$ is a Borel function related to the geometry of $\Omega$.  In this article we prove the following theorem.

\medskip

\begin{thm}\label{thm:main}
Let $k\in(0,1)$ and d%, consider the stretch map $f_k$ which is a sense preserving quasiconformal map
%from the Heisenberg spherical ring $S_{a,b}$ onto the Heisenberg spherical ring
%$S_{a^k,b^k}$ with maximal distortion $K_{f_k}=k^{-2}$. D
enote by $\mathcal{F'}$ the class of $\calC^2$ contact quasiconformal maps $f:S_{a,b}\to S_{a^k,b^k}$
which preserve the axis $(0,t)$ and map the boundary components of $S_{a,b}$ to the respective boundary
components of $S_{a^k,b^k}$. Then up to composition with rotations around the vertical axis the stretch map $f_k$ minimises the $2/3$-mean
distortion within the class $\mathcal{F}$: for any $f\in\mathcal{F}$ we have have
$$
k^{-1/3}=K_{f_k}^{1/6}=\fM_{2/3}(f_k,\rho_0)\le\fM_{2/3}(f,\rho_0),
%\frac{\iiint_{S_{a,b}}K_{f_k}^{2/3}(p)\rho_0^{4/3}(p) d\calL^3(p)}{\iiint_{S_{a,b}}\rho_0^{4/3}(p) d\calL^3(p)} \le
%\frac{\iiint_{S_{a,b}}K_f^{2/3}(p)\rho_0^{4/3}(p) %d\calL^3(p)}{\iiint_{S_{a,b}}\rho_0^{4/3}(p) d\calL^3(p)}
$$
where
$$
\rho_0(z,t)=\frac{1}{4\pi\log(b/a)}\cdot\frac{1}{|z|\cdot\left||z|^2-it\right|}\cdot\mathcal{X}
(S_{a,b})(z,t)\quad\text{for every}\quad (z,t)\in\fH.
$$
Moreover,
$
K_f\ge K_{f_k}^{1/4}.
$
\end{thm}

\medskip

To prove this theorem we develop a method of modulus of surfaces as a counterpart of the modulus method of curve families that is developed in \cite{BFP1}. Throughout  this paper, a {\it surface} $\calS$ inside $\fH$ is supposed in the Euclidean sense and endowed with sufficient regularity (i.e., $\calC^2$). Outside an exceptional set of $\calS$ of small Lebesgue measure (the {\it characteristic locus of $\E$}), a horizontal vector field along $\E$ is canonically defined; this is the {\it horizontal normal vector field $N^h_\calS$ of $\E$}. That is, the vector field $N^h_\calS$ of $\E$ lies in the horizontal subbubdle of $\fH$ at points of $\calS$. The following property (see Proposition \ref{N-ineq}) holds for surfaces $\calS$ and ${\widetilde \calS}$ such that ${\widetilde \calS}=f(\calS)$ where $f=(f_I,f_3)$ is a $\calC^2$ contact quasiconformal transformation:
 $$
 \lambda(|Zf_I|-|\overline{Z}f_I|)\cdot\|N^h_\calS\|\le \|N^h_{\widetilde \calS}\|\le\lambda(|Zf_I|+|\overline{Z}f_I|)\cdot\|N^h_\calS\|.
 $$
 Where, $\lambda$ is the square root of the Jacobian $J_f$ of $f$. A transformation $f$ is said to have the {\it minimal stretching property (mSP)} for $\E$ if the left hand side inequality is attained as an equality.

Given now a family $\Sigma$ of such surfaces inside a domain $\Omega\subset\fH$, we define in Section \ref{sec:modsurfaces} the modulus ${\rm Mod}(\Sigma)$ as
 $$
 {\rm Mod}(\Sigma)=\inf_{\rho\in{\rm Adm}(\Sigma)}\iiint_\Omega\rho^{4/3}d\calL^3.
 $$
 Here, ${\rm Adm}(\Sigma)$ is the set of positive Borel functions $\rho$ defined in $\Omega$ such that
 $$
 \iint_\calS\rho dS^h\ge 1.
 $$
 The integral on the left is defined in a local parametrisation $\sigma:U\to\R^3$, $(u,v)\mapsto\sigma(u,v)$ as
 $$
 \iint_U\rho(\sigma(u,v))\|N^h_\sigma(u,v)\|dudv.
 $$
 The following Modulus Inequality (Theorem \ref{thm-modulusinequalities}) holds: if $f:\Omega\to\Omega'$ is a $\calC^2$ contact quasiconformal transformation between domains of $\fH$ with distortion function $K_f(p)$ and $\Sigma$ is a family of surfaces inside $\Omega$ then
 $$
 {\rm Mod}(f(\Sigma))\le\iiint_\Omega K_f^{2/3}(p)\rho^{4/3}(p)d\calL^3(p),
 $$
 for every $\rho\in{\rm Adm}(\Sigma)$. From this, we also obtain that for an extremal density $\rho_0$ of $\Sigma$, i.e., for a $\rho_0\in{\rm Adm}(\Sigma)$ such that ${\rm Mod}(\Sigma)=\iiint_\Omega\rho_0^{4/3}d\calL^3$ we have
 $$
 \frac{{\rm Mod}(f(\Sigma))}{{\rm Mod}(\Sigma)}\le\fM_{2/3}(f,\rho_0)
 $$
 and also, if $K_f$ is the maximal distortion of $f$ then
 $$
 K_f^{-2/3}{\rm Mod}(\Sigma)\le{\rm Mod}(f(\Sigma))\le  K_f^{2/3}{\rm Mod}(\Sigma).
 $$

\medskip

The modulus method we develop here concerns domains $\Omega$ of $\fH$ which are foliated in a particular way by a family $\Sigma_0$ of  surfaces inside $\Omega$, see Lemma \ref{lem:mod1}: in this case we find an extremal density $\rho_0$ for $\Sigma_0$ from which we may explicitly calculate ${\rm Mod}(\Sigma_0)$. Let now  $\mathcal{F}'$ be the class of all $\calC^2$ contact quasiconformal maps from $\Omega$ to another domain $\Omega'$ which satisfy certain boundary conditions. Suppose  additionally  that there exists an $f_0\in{\mathcal F}'$ which has the mSP for each surface of $\Sigma_0$ and is also  such that its distortion function $K_{f_0}(p)$ is constant on each leaf $\E\in\Sigma_0$, see Lemma \ref{lem:mod2}. Then, we can  calculate ${\rm Mod}(f_0(\Sigma_0))$ and show that it equals to $\iiint_\Omega K_{f_0}^{2/3}(p)\rho_0^{4/3}(p)d\calL^3(p)$. In summary,  our modulus method is given in the next theorem.

\medskip

\begin{thm}\label{thm:modmethod}
 Let $\Sigma_0$, $\rho_0$ and $f_0$ be such that they satisfy the assumptions of Lemmas \ref{lem:mod1} and \ref{lem:mod2}. Let $\Sigma\supseteq\Sigma_0$ be a family of $\calC^2$ surfaces in $\Omega$ so that $\rho_0\in{\rm Adm}(\Sigma)$ and consider the class $\mathcal{F}'$ of all $\calC^2$ contact quasiconformal transformations $f:\Omega\to\Omega'$ which satisfy
$$
{\rm Mod}(f_0(\Sigma_0))\le{\rm Mod}(f(\Sigma)).
$$
Then for all $f\in\mathcal{F}'$ we have
$$
\fM_{2/3}(f_0,\rho_0)\le\fM_{2/3}(f,\rho_0).
$$
%$$
%\iiint_\Omega K_{f_0}^{2/3}(p)\rho_0^{4/3}(p)d\mathcal{L}^3(p)\le \iiint_\Omega K_{f}^{2/3}(p)\rho_0^{4/3}(p)d\mathcal{L}^3(p).
%$$
\end{thm}

\medskip

In the proof of Theorem \ref{thm:main} our family $\Sigma_0$ consists of Heisenberg cones $\calC_\alpha$, i.e., paraboloids with cartesian equation $t=-\alpha|z|^2$. These surfaces foliate $S_{a,b}$ in a manner according to the assumptions of Lemma \ref{lem:mod1}; subsequently,  in Proposition \ref{moduluscones} we find the density $\rho_0$ as in Theorem \ref{thm:main} and also that 
$$
{\rm Mod}(\Sigma_0)=\left(2^5\pi\log(b/a)\right)^{-1/3}{\rm B}(1/2,1/6).
$$
Now, for $k\in(0,1)$ the stretch map $f_k$ satisfies the assumptions of Lemma \ref{lem:mod2}, i.e., it has the mSP for $\Sigma_0$ and it is constant at each leaf. Thus we find
$$
{\rm Mod}(f_k(\Sigma_0))=\iiint_{S_{a,b}}K_{f_k}^{2/3}(p)\rho_0^{4/3}(p)d\mathcal{L}^3(p)=k^{-1/3}\cdot{\rm Mod}(\Sigma_0).
$$ 
In order to apply Theorem \ref{thm:modmethod} and conclude the proof, we must detect a surface family $\Sigma$ such that the density $\rho_0$ is admissible for $\Sigma$ and $f_k(\Sigma_0))\subseteq f(\Sigma)$ for each $f\in\mathcal{F'}$. Indeed, we find that this family $\Sigma$ consists of surfaces inside $S_{a,b}$ which join the two pieces of the boundary and enjoy an additional property: for every $\E\in\Sigma$, there is no  non-characteristic point  $p\in\E$ with a neighborhood $U_p$ being a piece of a Heisenberg sphere.

\medskip

At this point, we wish to make the following comment on the general setup of this work and in particular, on the proof of our main theorem. Having assumed $\calC^2$ regularity for our surfaces lead us to find a minimiser of the $2/3$-mean distortion $\fM_{2/3}(f,\rho_0)$ {\it only} among the class $\mathcal{F'}$ of $\calC^2$ contact quasiconformal mappings $f$ between spherical rings $S_{a,b}$ and $S_{a^k,b^k}$ of $\fH$ which map boundary components to respective boundary components. Of course, this class is smaller than the class $\mathcal{F}$ of arbitrary quasiconformal transformations between these two domains which satisfy the same boundary conditions.  It seems plausible to conjecture that the same method may apply in the case of {\it $\fH$-regular surfaces} (for the definition and the properties of these surfaces, see for instance \cite{KSC}). This will enable us to omit the regularity hypothesis for the quasiconformal mappings of our theorem and prove in fact that the stretch map is  a minimiser 
within the wider class $\mathcal{F}$.

\medskip

The paper is organised as follows. Preliminaries about the Heisenberg group $\fH$ and quasiconformal mappings defined on $\fH$ are given in Section \ref{sec:prel}. Section \ref{sec:surfaces} is brief study of $\calC^2$ regular surfaces embedded in $\fH$.% and in Section \ref{sec:qsurfaces} we study the effect of a $\calC^2$ contact quasiconformal transformation of $\fH$ which maps a regular surface to another. 
The definition of the modulus of a surface family, the Modulus Inequality as well as the modulus method and the proof of Theorem \ref{thm:modmethod} are in Section \ref{sec:modulus}. Finally, the proof of our main Theorem \ref{thm:main} lies in Section \ref{sec:main}.

\medskip

{\it Aknowledgments.} The author wishes to thank Zolt\'an Balogh and Katrin F\"assler for fruitful discussions and useful observations.% and also Prof. Theopi Parisaki for constant moral support over a long period of time. 

\section{Preliminaries}\label{sec:prel}
The material of this section is standard and we refer the reader to \cite{CDPT} for more details. The Heisenberg group is described in Section \ref{sec:heis} and a brief overview of the Kor\'anyi-Reimann theory of quasiconformal mappings of $\fH$ lies in Section \ref{sec:quas}.

\subsection{Heisenberg Group}\label{sec:heis}
The Heisenberg group $\fH$  is $\C\times\R$ whose group law and distance $d_\fH$  are described in the introduction.
%$$
%(z,t)*(w,s)=(z+w,t+s+2\Im(\overline{w}z)).
%$$
%The Heisenberg norm (Kor\'anyi gauge) is given by
%$$
%\left|(z,t)\right|=\left| |z|^2-it\right|^{1/2}.
%$$
%From this gauge we obtain  a metric, the Kor\'anyi--Cygan (K--C) metric on $\fH$, by the relation
%$$
%d_\fH\left((z_1,t_1),\,(z_2,t_2)\right)
%=\left|(z_1,t_1)^{-1}*(z_2,t_2)\right|.
%$$
The metric $d_\fH$ is invariant by a
 {\it left translations} $T_{(\zeta,s)}$, defined  for a given $(\zeta,s)\in\fH$ by
 $
 T_{(\zeta,s)}(z,t)$ $=(\zeta,s)*(z,t)
 $
and also by {\it rotations about the vertical axis} $R_\theta $, defined for a given $\theta\in\R$ by
 $
 R_\theta(z,t)=(ze^{i\theta},t).
 $
Left translations and rotations form the group ${\rm Isom}(\fH,d_\fH)$ of (orientaion-preserving) {\it Heisenberg isometries}. There are two other kinds of transformations that are of particular importance; in the first place we have {\it dilations} $D_\delta$  defined for a given $\delta>0$ by
 $
 D_\delta(z,t)=(\delta z, \delta^2 t).
 $
 One may check that  the metric $d_\fH$ is  scaled up to the multiplicative constant $\delta$ by the action of $D_\delta$. Finally there is an   {\it inversion} $I$ with respect to the unit Heisenberg sphere  given for all $(z,t)\neq(0,0)$ by
$
I(z,t)=\left(\frac{z}{-|z|^2+it}\;,\;\frac{-it}{\left|-|z|^2+it\right|^2}\right).%,\;\;\text{if}\;(z,t)\neq o,\infty,\;\quad R(o)=\infty,\;R(\infty)=o.
$
%Inversion $R$ has the following property: for all $p=(z,t),p'=(z',t')\in\fH-\{(0,0)\}$ we have
%$$
%d_\fH(R(p),o)=\frac{1}{d_\fH(p,o)},\quad d_\fH(R(p),R(p'))=\frac{d_\fH(p,p')}{d_\fH(p,o)\;d_\fH(o,p')}.
%$$

The Heisenberg group $\fH$ is a 2-step nilpotent Lie group with  left invariant vector fields
\begin{eqnarray*}
X=\frac{\partial}{\partial x}+2y\frac{\partial}{\partial t},\quad Y=\frac{\partial}{\partial y}-2x\frac{\partial}{\partial t},
\quad T=\frac{\partial}{\partial t}
\end{eqnarray*}
and  the complex fields $Z$, $\overline{Z}$ as in the introduction are just
\begin{eqnarray*}
Z=\frac{1}{2}(X-i Y)%=
%\frac{\partial}{\partial z}+i\overline{z}\frac{\partial}{\partial t}
,\quad
\overline{Z}=\frac{1}{2}(X+i Y).%=
%\frac{\partial}{\partial\overline{z}}-iz\frac{\partial}{\partial t}.
\end{eqnarray*}
There is a non-trivial commutation relation $[X,Y]=-4T$ and the tangent space to $\fH $ is spanned by $X,Y,T$. % The
 %kernel of $\omega$, also called the {\it horizontal space} ${\rm H}(\fH)$  is spanned by the vector fields $Z,\overline{Z}$.
%The vector field $T$ %is called the Reeb vector field and
%satisfies %$[X,Y]=-4T$ (or equivalently,
%$[Z,\overline{Z}]=-2iT$ and  $\omega(T)=1$.
%The Levi form $d\omega$ defines a positive definite
%hermitian form on ${\rm H}$ since $d\omega(X,JY)>0$. Thus the $\C\R-$structure $\omega$ is strictly pseudoconvex
%and thus contact.
%The volume form on $\fH$
%is $\omega\wedge d\omega=4dx\wedge dy\wedge dt, $ i.e., a multiple of the usual volume form obtained by the Lebesgue measure in $\C\times \R$.
 %and by considering the aforementioned commutation relations of its left invariant vector fields $X,Y,T$ (or, $Z,\overline{Z},T$).
The Lie algebra of left invariant vector fields of $\fH$ has a grading $\mathfrak{h} = \mathfrak{v}_1\oplus \mathfrak{v}_2$  with
\begin{displaymath}
\mathfrak{v}_1 = \mathrm{span}_{\R}\{X, Y\}\quad \text{and}\quad \mathfrak{v}_2=\mathrm{span}_{\R}\{T\}.
\end{displaymath}
The contact form $\omega$ of $\fH$ is defined as the unique 1-form satisfying $X,Y\in{\rm ker}\omega$, $\omega(T)=1$; in $(z,t)$ coordinates
$$
\omega=dt+2\Im(\overline{z}dz).
$$
Uniqueness here is modulo change of coordinates as it follows by the Darboux's Theorem. The distribution ${\rm H}(\fH)=\mathfrak{v}_1$ is called the {\it horizontal distribution}, any vector field $V\in\mathfrak{v}_1$ is called a {\it horizontal vector field} and if $p\in\fH$, the space ${\rm H}_p(\fH)=\mathrm{span}_{\R}\{X_p, Y_p\}$
is called the {\it horizontal tangent plane} to $p$. The {\it sub-Riemannian metric} of $\fH$ is the metric defined in ${\rm H}(\fH)$ by the relations $\langle X, X\rangle=\langle Y,Y\rangle=1$ and $\langle X,Y\rangle=\langle Y,X\rangle=0$; its norm shall be denoted by $\|\cdot\|$.

The {\it Legendrian foliation} is the foliation of $\fH$ by horizontal curves. An absolutely continuous curve $\gamma:[a,b]\to \fH$ (in the Euclidean sense) with 
%\begin{displaymath}
$\gamma(t)=(\gamma_I(t),\gamma_3(t))\in\mathbb{C}\times \mathbb{R}$
%\end{displaymath}
 is called {\it horizontal} if
\begin{displaymath}
 \dot{\gamma}(t)\in {\rm H}_{\gamma(t)}(\fH)\quad \text{for almost every}\;t\in [a,b].
\end{displaymath}
%For $\calC^1$ curves this condition reads as
A curve $\gamma:[a,b]\to \fH$ is absolutely continuous with respect to the  $d_\fH$ if and only if it is a horizontal curve. %Moreover,
%the length of a smooth rectifiable curve $\gamma=(\gamma_I,\gamma_3)$ with respect to $d_\fH$ is given by the integral over the (Euclidean) norm of the horizontal part of the tangent vector,
%\begin{displaymath}
% \ell(\gamma)=\int_a^b |\dot{\gamma}_I(t)|\;dt,
%\end{displaymath}
%see \cite{KR2}.

%Contact transformations between domains of $\fH$ play an important role in the theory of quasiconformal mappings of $\fH$.
A {\it contact transformation} $f:\Omega \to \Omega'$ on $\fH$ is a diffeomorphism between domains $\Omega$ and $\Omega'$ in $\fH$ which preserves the contact structure, i.e.,
\begin{equation}\label{eq:contact_form}
 f^*\omega = \lambda \omega,
\end{equation}
for some non-vanishing real valued function $\lambda$. We  write $f=(f_I,f_3)$, $f_I=f_1+\mathrm{i}f_2$. Then a contact map $f$ is completely determined by $f_I$ in the sense that the contact condition (\ref{eq:contact_form}) is equivalent to the following system of differential equations:
\begin{eqnarray}
\label{eq:C1}
&&\overline{f}_IZ f_I - f_I Z \overline{f}_I +iZf_3=0\\
&&
\label{eq:C2}
f_I \overline{Z}\overline{f}_I-\overline{f}_I \overline{Z}f_I -
i\overline{Z}f_3=0\\
&&
\label{eq:C3}
-i(\overline{f}_I T f_I - f_I T \overline{f}_I +
iTf_3)=\lambda.
\end{eqnarray}
If $f$ is $\mathcal{C}^2$ then %it is elementary to prove that
$
\det J_f=\lambda^2,
$
where by $J_f$ we denote the usual Jacobian matrix  of $f$.

\subsection{Quasiconformal Mappings}\label{sec:quas}

There are various analytic definitions of quasiconformality in $\fH$ which are all equivalent to the metric definition \ref{defn:metric}; for instance, we refer the reader to \cite{D}, \cite{H1} and \cite{V}. For the analytic definition of qc mappings we state here,  we follow the lines of \cite{KR1} and \cite{KR2} with minor deviations.  We  first   recall the notions of the $P$-{\it differentiability} of mappings between domains of $\fH$ and that of the {\it absolute continuity in lines} (ACL). According to  Pansu, \cite{Pa}, a mapping $f:\Omega\to\Omega'$ between domains of $\fH$ is called $P$-differentiable at $p\in\Omega$ if for $c\to 0$ the mappings
$$
D_c^{-1}\circ T_{f(p)}^{-1}\circ f\circ T_{p}\circ D_c
$$
converge locally uniformly to a homomorphism $(D_0)f_p$ from $T_p(\fH)$ to $T_{f(p)}(\fH)$  which preserves the horizontal space ${\rm H}(\fH)$. Here $D$ and $T$ are dilations and left translations respectively.
In terms of the standard basis $Z,\overline{Z},T$, the $P$-derivative of $f=(f_I,f_3)$ at $p$ is in matrix form
\begin{equation*}
 (D_0)f_p=\left(\begin{matrix} Zf_I & \overline{Z}f_I& 0\\
                 Z\overline{f_I}&\overline{Zf_I}&0\\
                 0&0&|Zf_I|^2-|\overline{Z}f_I|^2
                \end{matrix}\right)_p,
\end{equation*}
where all derivatives are in the distributional sense. Quasiconformal mappings between domains in $\fH$ are a.e. $P$-differentiable, see \cite{Pa}. In particular, (see Proposition 6 of \cite{KR2}), if $f$ is $P$-differentiable at $p\in\fH$ with derivative $(D_0)f_p$, then the restriction of $f$ to the plane
$$
\left\{p\exp(xX+yY)\;|\;(x,y)\in\R^2\right\}
$$
is differentiable at $p$ in the Euclidean sense and its derivative $(D_h)F_p$ is the restriction of $(D_0)F_p$ in horizontal spaces;  in matrix form, it is given by
\begin{equation*}
 (D_h)f_p=\left(\begin{matrix} Zf_I & \overline{Z}f_I\\
                 Z\overline{f_I}&\overline{Zf_I}\\
                \end{matrix}\right)_p,
\end{equation*}
and the following also hold for a $K$-quasiconformal mapping $f$ which is  $P$-differentiable at $p$ :
\begin{enumerate}
 \item $\|(D_h)f_p\|:=\max\left\{\|(D_h)f_p(V)\|\;|\; \|V\|=1\right\}=|Zf_I(p)|+|\overline{Z}f_I(p)|$ a.e.;
 \item $J_f(p)=\det (D_0)f_p=(\det (D_h)f_p)^2=\left(|Zf_I(p)|^2-|\overline{Z}f_I(p)|^2\right)^2$;
 \item $$
 K_f(p)^2=\frac{\|(D_h)f_p\|^4}{J_f(p)}=\left(\frac{|Zf_I(p)|+|\overline{Z}f_I(p)|}{|Zf_I(p)|-|\overline{Z}f_I(p)|}\right)^2\le K.
 $$
\end{enumerate}
The function $\Omega\ni p\to K_f(p)\in [1,\infty)$ is the {\it distortion function} of $f$ and the constant  $K_f=K^{1/2}$ is also called the {\it maximal distortion} of $f$.

\medskip

%The basic property concering the regularity of quasiconformal mappings on the complex plane (and more generally, on Euclidean spaces of arbitrary dimension) is absolute continuity in lines (ACL):  mappings with this property are absolutely continuous on a.e. fiber of any smooth fibration.

A mapping between domains in the Heisenberg group $\fH$ is called absolutely continuous in lines (ACL), if it is absolutely continuous  on almost all fibers of smooth {\it horizontal} fibrations. For such a fibration, the fibers $\gamma_p$ can be parametrised by the flow $f_s$ of a horizontal {\it unit} vector field $V$: i.e., $V$ is of the form $aX+bY$ with $|a|^2+|b|^2=1$. Mostow proved (see Theorem A in \cite{KR2}) that quasiconformal mappings are absolutely continuous on a.e. fiber $\gamma$ of any give fibration $\Gamma_V$ determined by a left invariant horizontal vector field $V$.

\medskip

Beltrami equations are in order. According to Theorem C in \cite{KR2}, if $f=(f_I,f_3)$ is an orientation preseving $K$-quasiconformal mapping between domains $\Omega$ and $\Omega'$ in $\fH$ then it satisfies a.e. the Beltrami type system of equations
\begin{eqnarray}\label{eq:B1}
&&
\overline{Z}f_I=\mu Zf_I,\\
&&\label{eq:B2}
\overline{Z}f_{II}=\mu Zf_{II},
\end{eqnarray}
where $f_{II}=f_3+i|f_I|^2$ and $\mu$ is a complex function in $\Omega$ such that
$$
\frac{1+\|\mu\|_\infty}{1-\|\mu\|_\infty}\le K \quad \text{a.e.}
$$
where $\|\mu\|_\infty={\rm ess sup}\{|\mu(z,t)|\;|\;(z,t)\in\Omega\}$. For each $p=(z,t)\in\Omega$, the function
\begin{equation*}
 \mu(p)=\mu_f(p)=\frac{\overline{Z}f_I(p)}{Zf_I(p)},
\end{equation*}
is called the {\it Beltrami coefficient of $f$}. If $K_f$ is the maximal distortion and $K_f(p)$ is the distortion function of $f$ respectively, then the following hold:
$$
|\mu_f(p)|=\frac{K_f(p)-1}{K_f(p)+1},\quad K_f(p)=\frac{1+|\mu_f(p)|}{1-|\mu_f(p)|},\quad \|\mu_f\|_\infty=\frac{K_f-1}{K_f+1}.
$$
We now state the analytic definition of quasiconformality in $\fH$.

\medskip

\begin{defn}\label{defn:analytic}
{\bf (Analytic definition)} A homeomorphism $f:\Omega\to\Omega'$, $f=(f_I,f_3)$, between domains in $\fH$ is an orientation-preserving quasiconformal mapping if
\begin{enumerate}
\item[{(i)}] it is ACL;
\item[{(ii)}] it is a.e. $P$-differentiable, and
\item[{(iii)}] it satisfies a.e. a system of Beltrami equations of the form \ref{eq:B1}, \ref{eq:B2}
where  $\mu$ is a complex function in $\Omega$ such that $\|\mu\|_\infty<1$.
\end{enumerate}
\end{defn}
An analogous definition holds for orientation reserving quasiconformal mappings.

\medskip

In the present paper we are considering quasiconformal maps with sufficient smoothness; these have to be contact transformations. %This property distinguishes quasiconformal mappings of the Heisenberg group $\fH$ from those defined on Euclidean spaces. In fact, f
From $P$-differentiability of  quasiconformal maps it follows that $P$-diffeomorphic $K$-quasiconformal mappings are contact transformations satisfying
\begin{equation}\label{cond:dil}
 \|(D_h)f\|^4\le K |J_f|\quad \text{a.e.}
\end{equation}
(Here, the absolute value in the Jacobian covers both situations of orientation-preserving and orientation-reversing mappings). The converse is also true, see Proposition 8 in \cite{KR2}: if a $\mathcal{C}^2$ contact transformation $f$ satisfies condition (\ref{cond:dil}), then $f$ is $K$-quasiconformal. We conclude that $K$-quasiconformal diffeomorphism lie in the class of contact transformations. Due to the contact conditions (\ref{eq:C1}), (\ref{eq:C2}) and (\ref{eq:C3}), equation (\ref{eq:B1}) in the Beltrami system implies equation (\ref{eq:B2}).

\section{Surfaces in the Heisenberg Group}\label{sec:surfaces}
There is a rather large bibliography about surfaces in the Heisenberg group and their geometrical properties; the reader should see for instance \cite{CDPT} for a treatment of {\it hypersurfaces} of $\fH$ and the references given therein. Our treatment in this section is somewhat different, see \cite{P1}; instead of hypersurfaces, i.e., graphs of function with sufficient regularity, we study regular surfaces $\E$ (here, regular means $\calC^2$ regular in the Euclidean sense) via surface patches. Accordingly, we define the horizontal space $\mathbb{H}(\E)$, the horizontal normal $N^h_\E$ and the characteristic locus $\mathfrak{C}(\E)$ of $\E$ in Section \ref{sec:horsurf}. In Section \ref{sec:induced} we show that the pullback of the contact form $\omega$ of $\fH$ in a regular surface $\E$ defines a contact 1-form $\omega_\E$ on $\E$ whose kernel is $\mathbb{J}N_\E^h$; here $\J$ is the natural complex operator acting on $\mathbb{H}(\E)$. The integral curves of $\mathbb{J}N_\E^h$ are horizontal curves lying in 
$\E$. Next, the horizontal area of a regular surface $\E$ is defined in Section \ref{sec:horarea}. The definition here is via surface patches and in the case of hypersurfaces it agrees with the definition given in \cite{CDPT}. Finally, Section \ref{sec:qsurfaces} is devoted to the study of contact $\calC^2$ transformations of $\fH$ which map a regular surface to another. Proposition \ref{N-ineq} and Corollary \ref{cor:ineq} are crucial for our subsequent discussion.

\subsection{Regular Surfaces, Horizontal Normal}\label{sec:horsurf}

For clarity, we recall the notion of  {\it regular surface} of $\R^3$ (see, for instance, \cite{DC}): this is a countable collection of surface patches (local charts) $\sigma_\alpha:U_\alpha\rightarrow V_\alpha\cap\R^3$ where $U_\alpha$ and $V_\alpha$ are open sets of $\R^2$ and $\R^3$, respectively, such that
\begin{enumerate}
 \item each $\sigma_\alpha$ is a $\calC^2$ homeomorphism, and
\item the differential $(\sigma_\alpha)_*:\R^2\rightarrow\R^3$ is of rank 2 everywhere.
\end{enumerate}
The tangent plane $T_\sigma(\E)$ of $\E$ at a surface patch $\sigma$ defined in an open domain $U\subset\R^2$ by
$$
\sigma(u,v)=(x(u,v),y(u,v),t(u,v))
$$
is
$
T_\sigma(\E)={\rm span}\left\{\sigma_u=\sigma_*\frac{\partial}{\partial u},\;\sigma_v=\sigma_*\frac{\partial}{\partial v}\right\}.
$
This may also be defined by the normal vector
\begin{eqnarray*}
N_\sigma=\sigma_u\wedge\sigma_v%=\sigma_*\frac{\partial}{\partial u}\wedge \sigma_*\frac{\partial}{\partial v}
=\frac{\partial(y,t)}{\partial(u,v)}\frac{\partial}{\partial x}
+\frac{\partial(t,x)}{\partial(u,v)}\frac{\partial}{\partial y}
+\frac{\partial(x,y)}{\partial(u,v)}\frac{\partial}{\partial t},
\end{eqnarray*}
where $\wedge$ is the exterior product in $\R^3$. We have
$
T_\sigma(\E)=\{V_\sigma\in T_\sigma(\R^3)\;:\;N_\sigma\cdot V_\sigma=0\}
$
where the dot is the usual Euclidean product in $\R^3$.

A regular surface $\E$ is {\it oriented} if  for every two overlapping patches $(U,\sigma)$ and $(\tilde U,\tilde\sigma)$  the transition mapping $\Phi=\sigma^{-1}\circ\tilde\sigma$ has positive Jacobian deterninant $J_\Phi=\det(\Phi_*)$. In $U\cap\tilde U$ we then have
$
N_{\tilde\sigma}^h=J_\Phi N_\sigma^h
$
and the unit normal vector field of $\nu_\E$ of $\E$ is uniquely defined at each local chart by the relation
$$
\nu_\sigma=\frac{\sigma_u\wedge\sigma_v}{|\sigma_u\wedge\sigma_v|},
$$
where $|\cdot|$ is the Euclidean norm in $\R^3$.

\medskip
{\it From now on, by a regular surface in $\fH$ we shall always mean a regular oriented surface in $\R^3$.}
\medskip

\begin{defn}\label{defn-hornorspace}
Let $\E$ be a regular surface and $p\in\E$. The {\sl horizontal plane} $\mathbb{H}_p(\E)$ of $\E$ at $p$ is the horizontal plane $\mathbb{H}_p(\fH)$.
\end{defn}

\medskip

Next, we define the horizontal normal vector $N^h_p$ at an arbitrary $p\in\E$. To do so, we first define
the Heisenberg wedge product  $\wedge^\fH_p$; this is the linear mapping $T_p(\fH)\times T_p(\fH)\to T_p(\fH)$ which assigns to each two vectors $a,b\in T_p(\fH)$, where $a=a_1X+a_2Y+a_3T$ and $b=b_1X+b_2Y+b_3T$,
 the  vector $a\wedge^\fH b\in T_p(\fH)$ which is  given by the formal determinant
\begin{equation*}
a\wedge^\fH b=\left|\begin{matrix}
                 X&Y&T\\
a_1&a_2&a_3\\
b_1&b_2&b_3     \end{matrix}\right|=\left|\begin{matrix}
a_2&a_3\\
b_2&b_3\end{matrix}\right|X+\left|\begin{matrix}
a_3&a_1\\
b_3&b_1\end{matrix}\right|Y+\left|\begin{matrix}
a_1&a_2\\
b_1&b_2\end{matrix}\right|T.
\end{equation*}
One may check that $\wedge^\fH_p$ is skew-symmetric and that following clock rule holds.
\begin{equation*}\label{clock}
 X\wedge^\fH Y=T,\quad Y\wedge^\fH T=X, \quad T\wedge^\fH X=Y.
\end{equation*}

\medskip

\begin{defn}
If $\sigma:U\rightarrow\R^3$ is a surface patch of a regular surface $\E$, the {\sl horizontal normal} $N^h_\sigma$ to $\sigma$ is the horizontal part of
$
\sigma_u\wedge^\fH \sigma_v=\sigma_*\partial/\partial u\wedge^\fH \sigma_*\partial/\partial v,
$
that is,
\begin{equation}\label{hornor2}
N_\sigma^h=(\sigma_u\wedge^\fH \sigma_v)^h=\sigma_u\wedge^\fH \sigma_v-\omega\left(\sigma_u\wedge^\fH \sigma_v\right)T.
\end{equation}
\end{defn}

\medskip

We stress here that the horizontal normal  $N^h_\sigma$ is {\sl not} the horizontal part of the normal $N_\sigma$. Simple calculations induce the following explicit formula:
\begin{equation}\label{defn-hornor}
N^h_\sigma=\left(\frac{\partial(y,t)}{\partial(u,v)}+2y\frac{\partial(x,y)}{\partial(u,v)}\right)X+
\left(\frac{\partial(t,x)}{\partial(u,v)}-2x\frac{\partial(x,y)}{\partial(u,v)}\right)Y.
\end{equation}
The horizontal normal $N_p^h$ at a point $p\in\E$ depends on the choice of the surface patch in the following way: suppose that $(U,\sigma)$ and $(\tilde U,\tilde\sigma)$ are two overlapping patches at $p$. Then if $\Phi=\sigma^{-1}\circ\tilde\sigma$ is the transition mapping, we may find from (\ref{hornor2})  that around $p$ we have
$
N_{\tilde\sigma}^h=J_\Phi N_\sigma^h,
$
where $J_\Phi={\rm det}(\Phi_*)>0$ since  $\E$ is oriented.

\medskip

\begin{defn}
Let $\E$ be a regular surface. A point $p\in\E$ is called {\sl non-characteristic} if $N_p^h\neq 0$. The set of characteristic points
$
\mathfrak{C}(\E)=\{p\in\E\;:\;N_p^h=0\}
$
is called the {\sl characteristic locus} of $\E$.
\end{defn}

\medskip

By definition, the points of $\mathfrak{C}(\E)$ are given in a local chart $(U,\sigma)$ by the equations
\begin{equation*}
\frac{\partial(y,t)}{\partial(u,v)}+2y\frac{\partial(x,y)}{\partial(u,v)}=0\quad\text{and}\quad
\frac{\partial(t,x)}{\partial(u,v)}-2x\frac{\partial(x,y)}{\partial(u,v)}=0.
\end{equation*}
%and it is proved that the Lebesgue measure of $\mathfrak{C}(\E)$ is 0 or 1.
An equivalent, but independent of coordinates  definition of the characteristic locus is given in Proposition \ref{prop-omega-locus} below.

\medskip

If $\sigma$ is a surface patch of $\E$ then the {\sl unit horizontal normal} $\nu^h_\sigma$ to $\sigma$ is defined at non-characteristic points by
\begin{equation}\label{unithornor2}
\nu^h_\sigma=\frac{N^h_\sigma}{\|N^h_\sigma\|} =\frac{(\sigma_u\wedge^\fH \sigma_v)^h}{\|(\sigma_u\wedge^\fH \sigma_v)^h\|},
\end{equation}
where $\|\cdot\|$ denotes the norm of the product $\langle\cdot,\cdot\rangle$ in $\mathbb{\fH}$ (recall that $\|X\|=\|Y\|=1$ and $\langle X,Y\rangle=0$). %We have
%\begin{equation}\label{unithornor}
 %\nu^h_\sigma=\frac{(\sigma_u\wedge^\fH \sigma_v)^h}{\|(\sigma_u\wedge^\fH \sigma_v)^h\|}.
%\end{equation}

\begin{corollary}
Let $\E$ be a regular surface of $\fH$. Then away from the characteristic locus, (\ref{unithornor2}) defines a nowhere vanishing vector field $\nu^h_\E\in\mathbb{H}(\E)$, such that $\|\nu^h_\E\|=1$.
\end{corollary}

\medskip

\noindent Associated to the horizontal normal vector field $\nu_\E^h$  there is a horizontal vector field induced by the complex operator $\J$  acting in $\mathbb{H}(\fH)$ by the relations
$
\mathbb{J}X=Y$ and $\mathbb{J}Y=-X.
$
Restricting this action  in the horizontal space  of a regular surface $\E$,  if $\nu^h_\E=\nu_1X+\nu_2Y$ then we have
$$
\J \nu^h_\E=-\nu_2X+\nu_1Y.
$$
%\medskip
%The vector fields $\nu_\E$ and $\J\nu_\E^h$ are of particular importance as we will show in the next section.

\subsection{Induced 1-form, Local Contactomorphisms, Horizontal Flow}\label{sec:induced}
If $\E$ is a regular surface in $\fH$ then a 1-form $\omega_\E$ may be defined in $\E$ in the following manner. Denote by $\iota_\E$ the inclusion map
$
\iota_\E:\E\hookrightarrow\fH,
$
given locally by a parametrisation $\sigma(u,v)=(x(u,v),y(u,v),t(u,v))$. If $\omega=dt+2xdy-2ydx$ is the contact form of $\fH$ then  $\omega_\E=\iota_\E^*\omega$; in the local parametrisation this is given by
\begin{eqnarray*}\label{omegaS}
 \omega_\E=\sigma^*\omega=(t_u+2xy_u-2yx_u) du+(t_v+2xy_v-2yx_v) dv.
\end{eqnarray*}

\medskip

We leave the proof of the following proposition to the reader.

\begin{prop}\label{prop-omega-locus}
% Thus
The characteristic locus $\mathfrak{C}(\E)$ is the (closed) set of points of $\E$ at which $\omega_\E=0$.
\end{prop}

\medskip

\begin{defn}\label{def:contact}
Let  $f:\calS\to\tilde\calS$ be a smooth diffeomorphism between regular surfaces $\calS$ and $\tilde\calS$ and outside the characteristic loci of $\calS$ and $\tilde\calS$. The mapping $f$ is called a {\it local contactomorphism of $\calS$ and $\tilde\calS$} if there exists a smooth function $\lambda$ so that
$
f^*\omega_{\tilde\calS}=\lambda\omega_\calS.
$
\end{defn}

\medskip

Since $f$ is a local diffeomorphism, if $\sigma:U\to\R^3$ is a surface patch for $\calS$ then $\tilde\sigma=f\circ\sigma$ is  a surface patch for $\tilde\calS$ (with the possible exception of characteristic points). It follows that $f:\calS\to\tilde\calS$ is a contactomorphism if and only if
\begin{equation}\label{eq:contact-cond}
\omega_{\tilde\sigma}(u,v)=\lambda(u,v)\omega_\sigma(u,v),\quad\text{for almost all}\;\;(u,v)\in U.
\end{equation}

\medskip

 Let now $\gamma:I\rightarrow\E$ be a  {\sl surface curve} on  a regular surface   $\E$,  that is a smooth mapping from an open interval of $\R$ to $\E$. The following proposition gives the condition under which a surface curve is horizontal, i.e., its horizontal tangent $\dot\gamma^h(s)\in\mathbb{H}_{\gamma(s)}(\E)$.

\medskip

\begin{prop}\label{horsurfacecurves}
Suppose that $\sigma:U\rightarrow\fH$ is a surface patch and $\gamma(s)=\sigma(u(s),v(s))$, $s\in I$, is a smooth surface curve (that is, $\tilde\gamma(s)=(u(s),v(s))$ a smooth curve in $U$). Then away from the characteristic locus, $\gamma$ is horizontal if and only if
$
\dot{\tilde\gamma}\in{\rm ker}\;\omega_\E,
$
or equivalently,
\begin{equation*}
 (t_u+2xy_u-2yx_u)\dot u+(t_v+2xy_v-2yx_v)\dot v=0
\end{equation*}
where the dot denotes $d/ds$.
In this case,
\begin{equation*}\label{hor-sur-cur}
 \dot\gamma=(x_u\dot u+x_v\dot v)X+(y_u\dot u+y_v\dot v)Y.
\end{equation*}
\end{prop}
\begin{proof}
We only prove the first statement. We have
\begin{eqnarray*}
\gamma\;\text{horizontal}&\Longleftrightarrow&\omega(\dot\gamma^h)=0\\
&\Longleftrightarrow&\omega(\sigma_*\dot{\tilde\gamma})=0\\
&\Longleftrightarrow&(\sigma^*\omega)(\dot{\tilde\gamma})=0\\
&\Longleftrightarrow&\dot{\tilde\gamma}\in{\rm ker}\;\omega_\E.
\end{eqnarray*}

\end{proof}

\medskip

The proof of the following proposition is in \cite{P1}. For clarity, we also sketch it here.
\medskip

\begin{prop}\label{integrability-Jn}
The 1-form $\omega_\E$ defines an integrable foliation of $\E$ (with singularities at characteristic points) by horizontal surface curves. These curves are tangent to $\J\nu_\E^h$.% where the latter is viewed as a vector field in $T(\E)$.
\end{prop}
\begin{proof}
Integrability is immediate: $\omega_\E$ is a 1-form defined in a two--dimensional manifold. % the first place we show integrability. Locally we have $d\omega_\E=d(\sigma^*\omega)=\sigma^*d\omega$
%and therefore
%$$
%\omega_\E\wedge d\omega_\E=\sigma^*\omega\wedge\sigma^* d\omega=\sigma^*(\omega\wedge d\omega)
%$$
%which is zero since $\omega\wedge d\omega$ is a 3--form. Thus by Frobenius' Theorem and Proposition \ref{horsurfacecurves} we have the first statement.
We set now
\begin{equation}\label{alphabeta}
\alpha=\frac{1}{\|N^h_\sigma\|}(t_u-2yx_u+2xy_u),\quad \beta=\frac{1}{\|N^h_\sigma\|}(t_v-2yx_v+2xy_v),
\end{equation}
where $\|N^h_\sigma\|=\|(\sigma_u\wedge^\fH\sigma_v)^h\|$. If
\begin{equation}\label{JV}
J\mathcal{V}=\beta \frac{\partial}{\partial u}-\alpha \frac{\partial}{\partial v}\in{\rm ker}\omega_\E,
\end{equation}
then straightforward calculations deduce $\sigma_*\mathcal(J\mathcal{V})=\J\nu_\E$.
The integral curves of $\J\nu_\E$ are the solutions of the system of differential equations
$
\dot u=\beta$ and $\dot v=-\alpha.
$
\end{proof}

\medskip

\begin{defn}\label{horflow}
The foliation of $\E$ by the integrable curves of $\J\nu_\E$ is called the {\sl horizontal flow} of $\E$.
\end{defn}

\subsection{Horizontal Area and Horizontal Area Integral}\label{sec:horarea}
In an arbitrary regular surface $\E$, the notion of the area $\A$ is given by integrating at each coordinate neighborhood $(U,\sigma)$ the length of the normal vector $N_\sigma=\sigma_u\times\sigma_v$. Accordingly, we define the {\it horizontal area} (elsewhere called the {\it perimeter}) of $\E$.
\begin{defn}\label{defn:horarea}
Let $\E$ be a regular surface in $\fH$ and suppose  that  $\sigma:U\rightarrow\fH$ is any surface patch. Let $N^h_\sigma=(\sigma_u\wedge^\fH\sigma_v)^h$. If $R$ is a domain in $U$ then, its {\sl horizontal area} is given by
\begin{equation}\label{horarea}
\A^h_\sigma(R)=\iint_R\|N^h_\sigma(u,v)\|dudv.
\end{equation}
\end{defn}
The above integral may of course be infinite; however, assuming that $R$ is contained in a rectangle whose closure lies inside $U$, then the integral is finite. Furthermore, a reparametrisation does not change the value of the integral. Finally, in the case where $\E$ is compact, the horizontal area of $\E$ is well defined and will be denoted by
$$
\A^h(\E)=\iint_\E d\E^h.
$$
Here $d\E^h$ is the {\sl horizontal area element} of $\E$; at each surface patch $(U,\sigma)$,  $$d\E^h=\|N^h_\sigma(u,v)\|dudv.$$

 With the assumptions of Definition \ref{defn:horarea} suppose also that $\rho:\E\to\R$ is a function. The {\it horizontal area integral of $\rho$} in $R$ is defined by
\begin{equation}
\iint_{\sigma(R)}\rho d\E^h=\iint_U\rho(\sigma(u,v))\|N^h_\sigma(u,v)\|dudv,
\end{equation}
if $\rho(\sigma(u,v)\|N^h_\sigma(u,v)\|\in L^1(R)$. Again, a reparametrisation does not change the integral and in the case where $\E$ is compact the horizontal area integral of $\rho$ is  defined globally as follows. Suppose that $\sigma_i:U_i\to\E$, $i\in I$ is a finite covering of $\E$ by surface patches and $\rho\sigma_i\|N^h_{\sigma_i}\|\in L^1(U_i)$ for each $i\in I$. Then
\begin{equation}\label{eq:horintegral}
\iint_\E\rho d\E^h=\sum_{i\in I}\iint_{U_i}\rho(\sigma_i(u_i,v_i))\|N^h_{\sigma_i}(u_i,v_i)du_idv_i.
\end{equation}

\subsection{Regular Surfaces and Contact-Quasiconformal Transformations}\label{sec:qsurfaces}
Let $\E$ and $\widetilde{\E}$ be two regular oriented surfaces in $\fH$. In this work we only consider mappings $\E\rightarrow\widetilde{\E}$ that are induced by $\calC^2$ orientation-preserving contact transformations $f=(f_1,f_2,f_3)$ of $\fH$: $f^*\omega=\lambda\omega$ where $\lambda=J_f^{1/2}>0$. Let  $f$ be such a transformation with the property $f(\E)=\widetilde{\E}$.
Since $f$ is a $\calC^2$ diffeomorphism and both $\E$ and $\widetilde{\E}$ are $\calC^2$ embedded submanifolds of $\fH$, it follows that the restriction $f_\E:\E\rightarrow\widetilde{\E}$ of $f$ to $\E$ is also a $\calC^2$ diffeomorphism between $\calS$ and $\widetilde\calS$. In particular, for every local charts $(U,\sigma)$ and $(\widetilde{U},\tilde\sigma)$ of $\E$ and $\widetilde{\E}$, respectively, the mapping $\tilde\sigma^{-1}\circ f\circ\sigma$ is a $\calC^2$ diffeomorphism in its domain of definition. Being also contact, the transformation $f$ adds something more to this, i.e., the surfaces $\calS$ and $\widetilde\calS$ are locally contactomorphic. %In fact, we have

\medskip

\begin{prop}\label{prop:surfcont}
Let $\E,\widetilde{\E}$ be two regular oriented surfaces in $\fH$ and $f=(f_1,f_2,f_3)$ be a $\calC^2$ orientation-preserving contact transformation of $\fH$ such  $f(\E)=\widetilde{\E}$. Then $\E$ and $\widetilde{\E}$ are locally contactomorphic.
\end{prop}
\begin{proof}
 If $\iota_\E$ and $\iota_{\widetilde{\E}}$ are the inclusions of  $\E$ and $\widetilde{\E}$ respectively in $\fH$ then
$
f\circ\iota_\E=\iota_{\widetilde{\E}}\circ f.
$
The result follows.
\end{proof}

\medskip

The next lemma is useful for our subsequent discussion.

\medskip

\begin{lem}\label{lem-Nfs}
Let $\E$ be a regular oriented surface in $\fH$ and $f=(f_1,f_2,f_3)$ be a $\calC^2$ orientation-preserving contact transformation of $\fH$ such that $f^*\omega=\lambda\omega$ where $\lambda=J_f^{1/2}$ and $J_f$ is the Jacobian determinant of $f$. %Xf_1Yf_2-Yf_1Xf_2$. 
Then the following hold.
\begin{enumerate}
\item If $(U,\sigma)$ is a surface patch of $\E$, then $(U,f\circ\sigma)$ is a surface patch for $\widetilde{\E}=f(\E)$.
     \item If $N_\sigma^h=n_1X+n_2Y$ is the horizontal normal vector of $\sigma$, then,
\begin{eqnarray}\label{Ndiff}&&
N_{f\circ\sigma}^h=\lambda\left((n_1Yf_2-n_2Xf_2)X+(n_2Xf_1-n_1Yf_1)Y\right).
\end{eqnarray}
\end{enumerate}
\end{lem}

\begin{proof}
The proof of (1) is immediate since the restriction of $f$ in $\E$ is a $\calC^2$ diffeomorphism. To prove (2) we first write the matrices of  $f_*$ and $\sigma_*$ with respect to the basis $\{X,Y,T\}$. Those  are
$$
\left(\begin{matrix}
           Xf_1&Yf_1&Tf_1\\
Xf_2&Yf_2&Tf_2\\
0&0&\lambda
          \end{matrix}\right)\quad\text{and}\quad
          \left(\begin{matrix}
          x_u&x_v\\
          y_u&y_v\\
          \alpha\|N^h_\sigma\|&\beta\|N^h_\sigma\| \end{matrix}\right),
$$
respectively, where $\alpha$ and $\beta$ are as in \ref{alphabeta}.
%Therefore
%\begin{eqnarray*}
% f_*N_\sigma^h&=&n_1f_*X+n_2f_*Y\\
%&=&n_1(Xf_1X+Xf_2Y)+n_2(Yf_1X+Yf_2Y)
%\end{eqnarray*}
%and formula \ref{diffN} follows.
%To prove \ref{Ndiff} we write first
%\begin{eqnarray*}
%&&
%\sigma_u=x_uX+y_uY+\alpha\|N^h_\sigma\|T,\\
% &&
%\sigma_u=x_vX+y_vY+\beta\|N_\sigma^h\|T,
%\end{eqnarray*}
Therefore, from chain rule we have
\begin{eqnarray*}
f_*\sigma_u%&=&x_uf_*X+y_uf_*Y+\alpha\|N^h_\sigma\|f_*T,\\
%&=&x_u(Xf_1X+Xf_2Y)+y_u(Yf_1X+Yf_2Y)+\alpha\|N^h_\sigma\|(Tf_1X+Tf_2Y+\lambda T),\\
&=&\left(x_uXf_1+y_uYf_1+\alpha\|N^h_\sigma\|Tf_1\right)X\\
&&+\left(x_uXf_2+y_uYf_2+\alpha\|N^h_\sigma\|Tf_2\right)Y\\
&&+\lambda\alpha\|N^h_\sigma\|T,
\end{eqnarray*}
and
\begin{eqnarray*}
f_*\sigma_v&=&\left(x_vXf_1+y_vYf_1+\beta\|N^h_\sigma\|Tf_1\right)X\\
&&+\left(x_vXf_2+y_vYf_2+\beta\|N^h_\sigma\|Tf_2\right)Y\\
&&+\lambda\beta\|N^h_\sigma\|T.
\end{eqnarray*}
The desired equation (\ref{Ndiff}) now follows from formula (\ref{hornor2}).%, and working in the same manner we may obtain \ref{diffJN}. and \ref{JNdiff}.
\end{proof}

\medskip

\begin{prop}\label{N-ineq}
With the hypotheses of Lemma \ref{lem-Nfs}, %Let $\E$ be an oriented regular surface in $\fH$ and $f=(f_I=f_1+if_2,f_3)$ be a $\calC^2$ contact orientation-preserving diffeomorphism. Then 
in surface patches $(U,\sigma)$ and $(U,f\circ\sigma)$ of $\E$ and $f(\E)$ respectively and at non-characteristic points, the following inequality holds.
\begin{equation}\label{eq-inequality}
 \lambda(|Zf_I|-|\overline{Z}f_I|)\|N^h_{\sigma}\|\le\|N^h_{f\circ\sigma}\|\le\lambda(|Zf_I|+|\overline{Z}f_I|)\|N_{\sigma}^h\|,
\end{equation}
where $\lambda=|Zf_I|^2-|\overline{Z}f_I|^2=J_f^{1/2}$ and $J_f$ is Jacobian determinant  of $f$.
\end{prop}
\begin{proof}
We engage complex terminology and we write $m=n_1+in_2$. In this manner,
\begin{eqnarray*}
 &&
n_1Yf_2-n_2Xf_2=\Re\left(m(Zf_I-Z\overline{f}_I)\right),\\
&&
n_2Xf_1-n_1Yf_1=\Im\left(m(Zf_I+Z\overline{f}_I)\right),
\end{eqnarray*}
and therefore, equation (\ref{Ndiff}) may be written equivalently as
\begin{eqnarray*}
\label{Ndiffcomplex} &&
 N_{f\circ\sigma}^h=2\lambda\Re\left((mZf_I-\overline{m}\overline{Z}f_I)\cdot Z\right)
\end{eqnarray*}
and subsequently,
\begin{eqnarray*}
\label{|Nfs|}&&
 \|N_{f\circ\sigma}^h\|=\lambda|Zf_I-e^{-2i\arg(m)}\overline{Z}f_I|\|N_\sigma^h\|.
\end{eqnarray*}
 Inequality (\ref{eq-inequality}) follows by applying the triangle inequality.
\end{proof}

\medskip

\begin{corollary}\label{cor:ineq}
 With the hypotheses of Proposition \ref{N-ineq}, suppose also that $f$ is quasiconformal with Beltrami coefficient $\mu$. Then:
 \begin{enumerate}
  \item The right inequality in (\ref{eq-inequality}) is attained as an equality if and only if
\begin{equation}\label{right-char}
\mu e^{-2i\arg(m)}<0,\;\text{equivalently}\;\arg{\mu}=\pi+2\arg{m}.
\end{equation}
\item The left inequality in (\ref{eq-inequality}) is attained as an equality if and only if
\begin{equation}\label{left-char}
\mu e^{-2i\arg(m)}>0,\;\text{equivalently}\;\arg{\mu}=2\arg{m}.
\end{equation}
 \end{enumerate}
\end{corollary}
\begin{proof}
If $f$ is quasiconformal with Beltrami coefficient $\mu$, then $\overline{Z}f_I/Zf_I=\mu,$ with $\mu$ essentially bounded by a constant less than 1. Therefore,
 \begin{equation*}
\|N^h_{f\circ\sigma}\|=\lambda|Zf_I||1-\mu e^{-2i\arg(m)}|\cdot\|N_\sigma^h\|
\end{equation*}
and inequality (\ref{eq-inequality}) may be written as
\begin{equation*}
 \left|1-|\mu|\right|\le \left|1-\mu e^{-2i\arg(m)}\right|\le \left|1+|\mu|\right|.
\end{equation*}
The proof follows.
\end{proof}

\medskip

In the case where %\ref{right-char}  we say that $f$ has the {\it maximal stretching property (MSP)} for $\E$ and
%in  case
 (\ref{left-char}) holds, we say that $f$ has the {\it minimal stretching property (mSP)} for $\E$.

 \section{Modulus of Surface Families, Modulus Inequality, Modulus method}\label{sec:modulus}
 In this section we define the modulus of a family $\Sigma$ of regular surfaces in $\fH$ (Section \ref{sec:modsurfaces}). The Modulus Inequality (Theorem \ref{thm-modulusinequalities}) is proved in Section \ref{sec:modinequality}. Finally the Modulus Method and the proof of Theorem \ref{thm:modmethod} are in Section \ref{sec:modmethod}.
 
\subsection{Modulus of Surface Families}\label{sec:modsurfaces}
  
By $\Sigma$ we shall denote a family of  regular surfaces in $\fH$. The set ${\rm Adm}(\Sigma)$ comprises of non-negative Borel functions $\rho$ in $\fH$ such that for every $\E\in\Sigma$ we have
$$
\iint_\E\rho d\E^h\ge 1.
$$

\medskip

\begin{defn}\label{defn-modulus}
The modulus of a family $\Sigma$ of regular surfaces in $\fH$  is defined by
$$
{\rm Mod}(\Sigma)=\inf_{\rho\in{\rm Adm}(\Sigma)}\iiint_\fH\rho^{4/3} d\mathcal{L}^3,
$$
where by $d\mathcal{L}^3$ we denote the Lebesgue measure in $\R^3$.
\end{defn}

\medskip

If the infimum is attained by a function $\rho_0\in{\rm Adm}(\Sigma)$, that is
$$
{\rm Mod}(\Sigma)=\iiint_\fH\rho_0^{4/3} d\mathcal{L}^3,
$$
then we call $\rho_0$ an {\it extremal density for $\Sigma$}.

\subsection{The Modulus Inequality}\label{sec:modinequality}
\begin{thm}\label{thm-modulusinequalities}
Let $\Omega$ and $\Omega'$ be domains in $\fH$ and $f:\Omega\rightarrow\Omega'$ be a $\calC^2$ orientation-preserving contact quasiconformal transformation. For any family of oriented regular surfaces inside $\Omega$ we have
\begin{equation}\label{modulusineq1}
{\rm Mod}(f(\Sigma))\le\iiint_\Omega K^{2/3}_f(p)\rho^{4/3}(p)d\mathcal{L}^3(p)\quad\text{for\;each}\quad \rho\in{\rm Adm}(\Sigma).
\end{equation}
 If moreover $K_f$ is the maximal distortion of $f$ ($K_f(p)\le K_f$ for all $p$), then
 \begin{equation}\label{modulusineq2}
  \frac{1}{K_f^{2/3}}{\rm Mod}(\Sigma)\le{\rm Mod}(f(\Sigma))\le K_f^{2/3}{\rm Mod}(\Sigma).
 \end{equation}

\end{thm}
\begin{proof}
For every $\rho\in{\rm Adm}(\Sigma)$ we define a non-negative Borel function in $\Omega'$ by the relation
\begin{equation*}\label{rhoprime}
\rho'=\left\{\begin{matrix}\frac{\rho}{\lambda(|Zf_I|-|\overline{Z}f_I|)}\circ f^{-1}&\;\text{in}\;\Omega,\\
\\ 0&\;\text{elsewhere},\end{matrix}\right.
\end{equation*}
where $\lambda=|Zf_I|^2-|\overline{Z}f_I|^2=J_f^{1/2}$. Then for each $\E\in\Sigma$ with $f(\E)=\E'$ we have by the left hand side of inequality (\ref{eq-inequality}) that
\begin{eqnarray*}
\iint_{\E'}\rho' d\E'^h\ge\iint_\E\rho d\E^h.
\end{eqnarray*}
Therefore by changing the variables $q=f(p)$ we obtain
\begin{eqnarray*}
\iiint_{\Omega'}(\rho')^{4/3}(q)d\mathcal{L}^3(q)&=&\iiint_{\Omega}(\rho'(f(p))^{4/3}J_f(p)d\mathcal{L}^3(p)\\
&=&\iiint_\Omega\rho^{4/3}(p)\left(\frac{|Zf_I(p)|+|\overline{Z}f_I(p)|}{|Zf_I(p)|-|\overline{Z}f_I(p)|}\right)^{2/3}d\mathcal{L}^3(p)\\
&=&\iiint_\Omega K^{2/3}_f(p)\rho^{4/3}(p)d\mathcal{L}^3(p).
\end{eqnarray*}
By taking the infimum over all functions in ${\rm Adm}(f(\Sigma))$ we obtain (\ref{modulusineq1}). Also the right hand side of (\ref{modulusineq2}) is obtained by (\ref{modulusineq1}) and the relation
$$
K_f(p)\le K_f\quad\text{for\; all}\quad p\in\Omega.
$$
To obtain the left hand side of (\ref{modulusineq2}) we consider the inverse transformation $f^{-1}:\Omega'\rightarrow\Omega$ which is also quasiconformal with maximal distortion $K_f$. Thus, by applying (\ref{modulusineq1}) we have
\begin{eqnarray*}
{\rm Mod}(\Sigma)={\rm Mod}(f^{-1}(\Sigma'))&\le&\iiint_{\Omega'}K^{2/3}_f(q)\rho^{4/3}(q)d\mathcal{L}^3(q)\\
&\le &K^{2/3}_f\iiint_{\Omega'}\rho^{4/3}(q)d\mathcal{L}^3(q)\quad\text{for\;all}\quad\rho\in{\rm Adm}(\Sigma')
\end{eqnarray*}
and the inequality follows after taking the infimum over all $\rho\in{\rm Adm}(\Sigma')$.
\end{proof}

\medskip

\begin{corollary}
 The modulus of surface families is a conformal invariant. %In addition, the modulus remain invariant under transformations which are conformal with respect to surfaces.
\end{corollary}
\subsection{The Modulus Method--Proof of Theorem \ref{thm:modmethod}}\label{sec:modmethod}

In this section we prove Theorem \ref{thm:modmethod}. We need the following two lemmas.

\begin{lem}\label{lem:mod1}
Let $\Omega$ be a domain in $\fH$, $J\subset\R$ and $U\subset\R^2$ be open sets. Suppose that
$
\Phi:U\times J\rightarrow \Omega$, $((u,v),\tau)\to\Phi(u,v,\tau)$
is a diffeomorphism which foliates $\Omega$ so that the following properties hold.
\begin{enumerate}
 \item For every $\tau\in J$, the surface patch $\sigma_\tau:U\rightarrow\Omega$ defined by  $\sigma_\tau(u,v)=\Phi(u,v,\tau)$ is  regular.
\item  The horizontal normal $N^h_{\sigma_\tau}$ of each $\sigma_\tau$  is not zero a.e..
\item There is a decomposition of the Lebesgue measure
$$
d\mathcal{L}^3(\Phi(u,v,\tau))=\|N^h_{\sigma_\tau}(u,v)\|^{4/3}dudvd\mu(\tau),
$$
where $d\mu(\tau)$ is a measure on $J$.
\end{enumerate}
Then, the function
\begin{equation*}
 \rho_0(p)=\left\{\begin{matrix}(|U|\cdot\|N^h_{\sigma_\tau}(u,v)\|)^{-1}\;&\;p=\sigma_\tau(u,v),\\
 \\
                   0\;&\;\text{elsewhere},
                  \end{matrix}\right.
\end{equation*}
where $|U|={\rm Area}(U)$, is extremal for the surface family $\Sigma_0=\{\sigma_\tau\;|\;\tau\in J\}$ and moreover
\begin{equation*}
 {\rm Mod}(\Sigma_0)=|U|^{-1/3}\int_J d\mu(\tau).
\end{equation*}
\end{lem}
\begin{proof}
The density $\rho_0$ is admissible for $\Sigma_0$:
$$
\iint_U\rho_0(\sigma_\tau(u,v))\|N^h_{\sigma_\tau}(u,v)\|dudv=|U|^{-1}\iint_Ududv=1
$$
and furthermore
\begin{eqnarray*}
\iiint_\Omega\rho_0^{4/3}(p)d\mathcal{L}^3(p)&=&\int_J\left(\iint_U\rho_0^{4/3}(\sigma_\tau(u,v))\|N^h_{\sigma_\tau}(u,v)\|^{4/3}dudv\right)d\mu(\tau)\\
&=&|U|^{-1/3}\int_Jd\mu(\tau).
\end{eqnarray*}
Thus
$$
{\rm Mod}(\Sigma_0)\le|U|^{-1/3}\int_Jd\mu(\tau).
$$
For the reverse equality, we consider an arbitrary density $\rho$ such that
$$
1\le\iint_U\rho\|N^h_{\sigma_\tau}(u,v)\|dudv
$$
and we apply H\"older's inequality:
\begin{eqnarray*}
1&\le&\left(\iint_U\rho^{4/3}\|N^h_{\sigma_\tau}(u,v)\|^{4/3}dudv\right)^{3/4}\left(\iint_U dudv\right)^{1/4}\\
\therefore \quad \iint_U\rho^{4/3}\|N^h_{\sigma_\tau}(u,v)\|^{4/3}dudv&\ge &|U|^{-1/3},\\
\therefore \quad \iiint_\Omega \rho^{4/3}d\mathcal{L}^3&\ge &|U|^{-1/3}\int_Jd\mu(\tau).
\end{eqnarray*}
The result follows after taking the infimum over all $\rho\in{\rm Adm}(\Sigma_0)$.
\end{proof}

\begin{lem}\label{lem:mod2}
 Let $f_0:\Omega\to\Omega'$ be an orientation-preserving quasiconformal diffeomorphism between domains in $\fH$. Let $\Phi$ be a foliation of $\Omega$ and $\Sigma_0$ be a family of surfaces as in Lemma \ref{lem:mod2}. Assume in addition that $f_o$ fas the mSP for the family $\Sigma_0$
%$$
%\Sigma_0=\{\sigma(\cdot,\tau)\;|\;\tau\in I\}
%$$
and that
$$
K_{f_0}(\sigma((u,v),\tau))\equiv K_{f_0}(\tau)
$$
for all $((u,v),\tau)\in U\times J$. Then
$$
{\rm Mod}(f_0(\Sigma_0))=|U|^{-1/3}\int_J K_{f_0}^{2/3}(\tau)d\mu(\tau)=\iiint_\Omega K_{f_0}^{2/3}(p)\rho_0^{4/3}(p)d\calL^3(p).
$$
\end{lem}
\begin{proof}
 Let $\rho'\in{\rm Adm}(f_0(\Sigma_0))$ be an arbitrary density and let $\tilde\sigma=f_0\circ\sigma$. For convenience we shall write $({\bf u},\tau)$ instead of $((u,v),\tau)$. We have
\begin{eqnarray*}
 1&\le&\iint_U\rho'(\tilde\sigma({\bf u},\tau))\|N^h_{\tilde\sigma}({\bf u})\|dudv\\
&=&\iint_U \rho'(\tilde\sigma({\bf u},\tau)))J_{f_0}^{1/2}(\sigma({\bf u},\tau))(|Z(f_0)_I|-|\overline{Z}(f_0)_I|)_{\sigma({\bf u},\tau)}\|N^h_{\sigma}({\bf u})\|dudv,
\end{eqnarray*}
since $f_0$ has the mSP for the family $\Sigma_0$. By applying H\"older's inequality we have
\begin{eqnarray*}
&&
|U|^{-1/3}\le \iint_U (\rho'(\tilde\sigma({\bf u},\tau)))^{4/3}J_{f_0}^{2/3}(\sigma({\bf u},\tau))(|Z(f_0)_I|-|\overline{Z}(f_0)_I|)_{\sigma({\bf u},\tau)}^{4/3}\|N^h_{\sigma}({\bf u})\|^{4/3}dudv.
\end{eqnarray*}

By taking advantage of the assumption that the distortion function is constant along the surfaces of $\Sigma_0$, we multiply both sides of the inequality by $ K_{f_0}(\tau)^{2/3}$ and integrate afterwards over $J$ with respect to $\mu$. Then,
$$
 |U|^{-1/3}\int_J K_{f_0}^{2/3}(\tau)d\mu(\tau)\le\int_J\iint_U(\rho'(f_0\circ\sigma)((u,v),\tau)))^{4/3}J_{f_0}(\sigma({\bf u},\tau))\|N^h_{\sigma}({\bf u})\|^{4/3}dudvd\mu(\tau),
$$
and by changing the variables, this yields
$$
|U|^{-1/3}\int_J K_{f_0}^{2/3}(\tau)d\mu(\tau)\le\iiint_{\Omega}(\rho'(f_0(p)))^{4/3}J_{f_0}(p)d\mathcal{L}^3(p)=\iiint_{\Omega'}(\rho'(q))^{4/3}d\mathcal{L}^3(q).
$$
Since $\rho'$ is arbitrary, we obtain
$$
{\rm Mod}(f_0(\Sigma_0))\ge|U|^{-1/3}\int_J K_{f_0}^{2/3}(\tau)d\mu(\tau).
$$

We consider now the  density $\rho_0'$ given by
$$
\rho'(q)=\left\{\begin{matrix}(|U|(J_{f_0}(\sigma({\bf u},\tau)))^{1/2}(|Z(f_0)_I|-|\overline{Z}(f_0)_I|)_{\sigma({\bf u},\tau))}\|N^h_{\sigma}({\bf u})\|)^{-1}\;&\;q=f(p),\;p=\sigma({\bf u},\tau),\\
 \\
                   0\;&\;\text{elsewhere}.
                  \end{matrix}\right.
$$
It is admissible for $f_0(\Sigma_0)$:
$$
\iint_U\rho'_0(\tilde\sigma({\bf u},\tau))\|N^h_{\tilde\sigma}({\bf u})\|dudv=|U|^{-1}\iint_U dudv=1.
$$
Hence,
\begin{eqnarray*}
 {\rm Mod}(f_0(\Sigma_0))&\le&\iiint_{\Omega'}(\rho_0'(q))^{4/3}d\mathcal{L}^3(q)=\iiint_\Omega(\rho_0'(f_0(p)))^{4/3}J_{f_0}(p)d\mathcal{L}^3(p)\\
&=&\int_I\iint_U(\rho_0'(\tilde\sigma({\bf u},\tau)))^{4/3}J_{f_0}(\tilde\sigma({\bf u},\tau))\|N^h_\sigma({\bf u})\|^{4/3}dudvd\mu(\tau)\\
&=&\int_I\iint_U|U|^{-4/3}K_{f_0}^{2/3}(\tau)dudvd\mu(\tau)=|U|^{1/3}\int_J K_{f_0}^{2/3}(\tau)d\mu(\tau).
\end{eqnarray*}
We show finally that
$$
{\rm Mod}(f_0(\Sigma_0))=\iiint_\Omega K_{f_0}^{2/3}(p)\rho_0^{4/3}(p)d\calL^3(p)
$$
by explicitly calculating the integral $I$ on the right:
\begin{eqnarray*}
I&=&|U|^{-4/3}\iiint_{U\times J}K_{f_0}^{2/3}(\tau)\|N^h_{\sigma_\tau}(u,v)\|^{-4/3}\cdot\|N^h_{\sigma_\tau}(u,v)\|^{-4/3}dudvd\mu(\tau)\\
&=&|U|^{-1/3}\int_JK_{f_0}^{2/3}(\tau)d\mu(\tau)\\
&=&{\rm Mod}(f_0(\Sigma_0)).
\end{eqnarray*}
\end{proof}

\medskip

\noindent{\it Proof of Theorem \ref{thm:modmethod}}. 
Since $\rho_0$ is admissible for the larger family $\Sigma$ we have ${\rm Mod}(\Sigma)={\rm Mod}(\Sigma_0)$. From our assumption ${\rm Mod}(f_0(\Sigma_0))\le{\rm Mod}(f(\Sigma))$ for each $f\in\mathcal{F}$ we have
$$
\fM_{2/3}(f_0,\rho_0)=\frac{{\rm Mod}(f_0(\Sigma_0))}{{\rm Mod}(\Sigma_0)}\le\frac{{\rm Mod}(f(\Sigma))}{{\rm Mod}(\Sigma)}\le\fM_{2/3}(f,\rho_0),
$$
where for the last inequality we have used the Modulus Inequality (\ref{modulusineq1}).
\hfill$\Box$

\section{Proof of the Main Theorem}\label{sec:main}
In this section we prove Theorem \ref{thm:main}. The proof relies heavily in the use of logarithmic coordinates for $\fH$; an overview is in Section \ref{sec:log}. We review in brief some known facts about the stretch map in Section \ref{sec:stretch}. Next, the proof is given in three steps. In Step 1 we calculate the modulus ${\rm Mod}(\Sigma_0)$ and the extremal density $\rho_0$ of our prototype family $\Sigma_0$ which consists of pieces of Heisenberg cones lying inside the spherical ring $S_{a,b}$. We do this using Lemma \ref{lem:mod1}, the foliation here is given by the logarithmic coordinates parametrisation $\Phi$. In the second step of the proof we show that $f_k$ satisfies the assumptions of Lemma \ref{lem:mod2} and in order to apply Theorem \ref{thm:modmethod} and conclude the proof, we show in the third step that there exists a  wider surface family $\Sigma$ inside $S_{a,b}$ such that $rho_0\in{\rm Adm}(\Sigma)$ and $f_k(\Sigma_0)\subseteq f(\Sigma)$ for each $f\in\mathcal{F}'$.

\subsection{Logarithmic coordinates}\label{sec:log}
Details about this section can be found in \cite{P} as well as in \cite{BFP1}. Logarithmic coordinates for $\fH$ are directly analogous to logarithmic coordinates in the complex plane. These are given by the map $\Phi:{\widetilde \fH}_0=\R\times (-\pi/2,\pi/2)\times\R\to\fH\setminus\mathcal{V}={\widetilde \fH}$ where $\mathcal{V}$ is the vertical axis, by the relation
\begin{equation}\label{eq:coords}
\Phi(\xi,\psi,\eta)=\left(i\cos^{1/2}\psi e^{\frac{\xi+i(\psi-3\eta)}{2}},-\sin\psi e^\xi\right).
\end{equation}
We state in brief some known facts about logarithmic coordinates. In the following, if $f:Q\to{\widetilde \fH}$ is a $\calC^k$ map, where $Q$ is a simply connected subset of ${\widetilde \fH}$ and $k\ge 0$, then $\tilde f$ will denote the map defined by the relation $\Phi\circ f=\tilde f\circ\Phi$. We shall write
$$
\tilde f(\xi,\psi,\eta)=(\Xi(\xi,\psi,\eta),\Psi(\xi,\psi,\eta),{\rm H}(\xi,\psi,\eta)).
$$
The contact form $\omega$ has the following expression in logarithmic coordinates:
$$
\omega=-e^\xi(\sin\psi d\xi+3\cos\psi d\eta),
$$
and the contact conditions (\ref{eq:C1}), (\ref{eq:C2}) and (\ref{eq:C3}) for a contact $f$ are equivalent to  the following conditions for $\tilde f$:
\begin{eqnarray}
&&\label{eq:contlog1}
{\rm H}_\psi+\frac{1}{3}\tan\Psi\cdot\Xi_\psi=0,\\
&&\label{eq:contlog2}
W_{\xi,\eta}{\rm H}+\frac{1}{3}\tan\Psi\cdot W_{\xi,\eta}\Xi=0,
\end{eqnarray}
where $W_{\xi,\eta}=\partial_\xi-(\tan\psi/3)\partial_\eta$.  We wish to stress at this point that 
$$
\Phi_*(W_{\xi,\eta})=\Re\left(\frac{|z|^2-it}{\overline{z}}Z\right),\quad\text{and}\quad \Phi_*(\partial_\psi)=\Im\left(\frac{|z|^2-it}{\overline{z}}Z\right),
$$
generate the horizontal tangent space at each point of $\widetilde{\fH}$.% which does not lie in the vertical axis $\mathcal{V}$.

Finally, the Beltrami coefficient for $f$ is
$$
\mu_f(\Phi(\xi,\psi,\eta))=-e^{3i(\psi-\eta)}\frac{{\overline W}(\Xi+i\Psi)}{W(\Xi+i\Psi)},
$$
where $W=W_{\xi,\eta}-i\partial_\psi$ and $\overline{W}=W_{\xi,\eta}+i\partial_\psi$.
%\footnote{These differ as following from the logarithmic coordinates $(\xi,\psi,\eta)$: $\xi$ is the same, $\theta=-\psi$ and $3\eta=\pi-\theta-2\phi$. Maybe in a subsequent version have to switch back.}

%By calculating directly, we find
%\begin{eqnarray*}
% &&
%x_\xi=\frac{1}{2}e^{\xi/2}\cos^{1/2}\theta\cos\phi,\quad x_\theta=-\frac{1}{2}e^{\xi/2}\cos^{-1/2}\theta\sin\theta\cos\phi,\quad x_\phi=-e^{\xi/2}\cos^{1/2}\theta\sin\phi,\\
%&&
%y_\xi=\frac{1}{2}e^{\xi/2}\cos^{1/2}\theta\sin\phi,\quad y_\theta=-\frac{1}{2}e^{\xi/2}\cos^{-1/2}\theta\sin\theta\sin\phi,\quad y_\phi=e^{\xi/2}\cos^{1/2}\theta\cos\phi,\\
%&&
%t_\xi=e^\xi\sin\theta,\quad t_\theta=e^\xi\cos\theta,\quad t_\phi=0,
%\end{eqnarray*}
%and thus,

\subsection{The stretch map}\label{sec:stretch}

The stretch map $f_k\fH\to\fH$, $k\in\R$,  $k\neq 0,1,$ is defined in logarithmic coordinates by
$$
\tilde{f}_k(\xi,\psi,\eta)=\left(k\xi,\tan^{-1}\left(\frac{\tan\psi}{k}\right),\eta\right).
$$
In \cite{BFP1} it is shown that $\tilde{f}_k$ satisfies the contact conditions (\ref{eq:contlog1}) and (\ref{eq:contlog2}) and that the Beltrami coefficient of $f_k$ is
$$
\mu_{f_k}(\Phi(\xi,\psi,\eta))=-e^{3i(\psi-\eta)}\frac{k^2-1}{k^2+1+\tan^2\psi}.
$$
Since $\|\mu_{f_k}\|_\infty<1$,  $f_k$ is quasiconformal in $\fH$. In the following, the domain of $f_k$ will be the spherical ring $S_{a,b}=\{(z,t)\in\fH\;:\;a^4<|z|^4+t^2<b^4\}$ and we  shall consider the  parametrisation of $S_{a,b}\setminus\mathcal{V}$  by logarithmic coordinates
\begin{equation*}\label{ring-parameters}
 (\xi,\psi,\eta)\mapsto\Phi(\xi,\psi,\eta),%\left(i\cos^{1/2}\psi e^{\frac{\xi+i(\psi-3\eta)}{2}},-e^\xi\sin\psi\right)
\end{equation*}
where $(\xi,\psi,\eta)\in(2\log a,2\log b)\times(-\pi/2,\pi/2)\times (-2\pi/3,2\pi/3)$. Since the Jacobian determinant of $\Phi$ equals to $J_\Phi=-(3/4)e^{2\xi}$,  the Lebesgue mesure in $S_{a,b}$ is
$$
d\mathcal{L}^3(\xi,\psi,\eta)=\frac{3}{4}e^{2\xi}d\xi d\psi d\eta.
$$

\medskip

We are now set for the proof of our main theorem.

\medskip

\noindent{\it Proof of Theorem \ref{thm:main}}. The proof shall be given in steps.

%\medskip

\subsubsection*{\bf{Step 1: Modulus of $\Sigma_0$}}\label{sec:step1}
Our  prototype family will be the  foliation $\Sigma_0$ of the spherical ring $S_{a,b}$ by the pieces of Heisenberg cones $\mathcal{C}_\alpha$ lying inside $S_{a,b}$; these have cartesian equations
$$
t=-\alpha(x^2+y^2),\quad \alpha\in\R,\quad a^2<x^2+y^2<b^2.
$$
In logarithmic coordinates those equations are $\tan\psi=\alpha$ where
$\psi\in$ $(-\pi/2,\pi/2)$ and thus a cone $\mathcal{C}_\alpha$  may be parametrised by the single patch 
$$
\sigma_\psi(\xi,\eta)=\Phi(\xi,\psi,\eta),\quad (\xi,\eta)\in U=(2\log a,2\log b)\times(-2\pi/3,2\pi/3).
$$
At each leaf $\calC_\alpha$ of the foliation $\Sigma_0$, the chart
$
\left(U,\sigma_\psi\right)
$
is  a $\calC^2$ chart in the atlas of the leaf. For the horizontal normal $N^h_{\sigma_\psi}$ we calculate
\begin{eqnarray*}
&&
 \frac{\partial(x,y)}{\partial(\xi,\eta)}=-\frac{3}{4}e^\xi\cos\psi,\\
&&
\frac{\partial(y,t)}{\partial(\xi,\eta)}=-\frac{3}{2}e^{3\xi/2}\cos^{1/2}\psi\sin\psi\sin((3\eta-\psi)/2),\\
&&
\frac{\partial(t,x)}{\partial(\xi,\eta)}=-\frac{3}{2}e^{3\xi/2}\cos^{1/2}\psi\sin\psi\cos((3\eta-\psi)/2).
\end{eqnarray*}
Therefore
\begin{eqnarray}
 &&
\label{hornorpsi}
 N^h_{\sigma_\psi}(\xi,\eta)=-\frac{3}{2}e^{3\xi/2}\cos^{1/2}\psi\left(\cos(3(\eta-\psi)/2)X-\sin(3(\eta-\psi)/2)Y\right),\\
&&\label{horareapsi}
 \|N^h_{\sigma_\psi}(\xi,\eta)\|=\frac{3}{2}e^{3\xi/2}\cos^{1/2}\psi.
\end{eqnarray}

\medskip

We can now use Lemma \ref{lem:mod1} to calculate the extremal density and the modulus of $\Sigma_0$.

\medskip
\begin{prop}\label{moduluscones}
 Let $\Sigma_0$ be the family of pieces of Heisenberg cones $\mathcal{C}_\alpha$  inside the spherical ring $S_{a,b}$. Then
\begin{equation}\label{mod-cones}
{\rm Mod}(\Sigma_0)=\left(2^5\pi\log(b/a)\right)^{-1/3}{\rm B}(1/2,1/6),%\frac{3\Gamma^3(1/3)}{8\pi^{4/3}}(\log(b/a))^{-1/3}
\end{equation}
where ${\rm B}$ denotes the beta function. The extremal density $\rho_0$ is  given by
\begin{equation}\label{cone-density}
 \rho_0(z,t)=\frac{1}{4\pi\log(b/a)}\cdot\frac{1}{|z|\left||z|^2-it\right|}\mathcal{X}(S_{a,b}).
\end{equation}
\end{prop}
\begin{proof}
We  apply Lemma \ref{lem:mod1}.  Here:
\begin{enumerate}
\item $U=(2\log a,2\log b)\times(-2\pi/3,2\pi/3)$, $J=(-\pi/2,\pi/2)$,
$
\Phi((\xi,\eta),\psi)=\Phi(\xi,\psi,\eta)%\left(i\cos^{1/2}\psi e^{\frac{\xi+i(\psi-3\eta)}{2}},-e^\xi\sin\psi\right)
$
and
\item $\Sigma_0$ is the family $\sigma_\psi(\xi,\eta)=\Phi((\xi,\eta),\psi)$.
\end{enumerate}
From (\ref{horareapsi}) we have $\|N^h_{\sigma_\psi}(\xi,\eta)\|=\frac{3}{2}e^{3\xi/2}\cos^{1/2}\psi$ and also
$$
d\mathcal{L}^3\left(\Phi\left((\xi,\eta),\psi)\right)\right)=\|N^h_{\sigma_\psi}(\xi,\eta)\|^{4/3}d\xi d\eta d\mu(\psi),
$$
where
$$
d\mu(\psi)=\frac{1}{2}\left(\frac{3}{2}\right)^{-1/3}\cos^{-2/3}\psi d\psi,\quad \psi\in(-\pi/2,\pi/2).
$$
Formulae (\ref{mod-cones}) and (\ref{cone-density}) follow immediately after short calculations.
\end{proof}

\subsubsection*{\bf{Step 2: Modulus of $f_k(\Sigma_0)$}}

The stretch map has the mSP for the  family of cones $\Sigma_0$. Indeed, from equation (\ref{hornorpsi}) we have
$$
m=m_{\sigma_\psi}=-\frac{3}{2}e^{3\xi/2}\cos^{1/2}\psi e^{3i(\psi-\eta)/2},\quad \arg(m)=\pi+\frac{3}{2}(\psi-\eta).
$$
Therefore
\begin{eqnarray*}
 \mu_{f_k}e^{-2i\arg(m)}&=&\frac{1-k^2}{k^2+1+\tan^2\psi},
\end{eqnarray*}
which is positive if $k\in(0,1)$. Thus the distortion function $K_{f_k}$ is
$$
K_{f_k}(\Phi(\xi,\psi,\eta))=\frac{1+\tan^2\psi}{k^2+\tan^2\psi}
$$
and it is constant in each leaf of $\Sigma_0$ since it only depends on $\psi$. By Lemma \ref{lem:mod2} we have
\begin{eqnarray*}
{\rm Mod}(f_k(\Sigma_0))&=&\left(\frac{8\pi}{3}\log(b/a)\right)^{-1/3}\int_{-\pi/2}^{\pi/2}\left(\frac{1+\tan^2\psi}{k^2+\tan^2\psi}\right)^{2/3}\cdot \frac{1}{2}\left(\frac{3}{2}\right)^{-1/3}\cos^{-2/3}\psi d\psi\\
&=&\left(2^5\pi\log(b/a)\right)^{-1/3}\int_{-\pi/2}^{\pi/2}\frac{1+\tan^2\psi}{(k^2+\tan^2\psi)^{2/3}}d\psi\\
(z=\tan\psi)\;&=&\left(2^5\pi\log(b/a)\right)^{-1/3}\int_{-\infty}^{+\infty}\frac{dz}{(k^2+z^2)^{2/3}}\\
(u=z/k)\;&=&k^{-1/3}\cdot\left(2^5\pi\log(b/a)\right)^{-1/3}\int_{-\infty}^{+\infty}\frac{du}{(1+u^2)^{2/3}}\\
(v=\arctan u)\;&=&k^{-1/3}\cdot\left(2^5\pi\log(b/a)\right)^{-1/3}\int_{-\pi/2}^{\pi/2}\cos^{-2/3}vdv\\
&=&k^{-1/3}\cdot{\rm Mod}(\Sigma_0).
\end{eqnarray*}
Note that this equality could have been immediately deduced by the fact that $f_k$ maps Heisenberg cones to Heisenberg cones and the formula for the modulus of the family of Heisenberg cones inside a spherical ring. On the other hand, if $\rho_0$ is given by (\ref{cone-density}), then again by Lemma \ref{lem:mod2} we have
\begin{eqnarray*}&&
 \iiint_{S_{a,b}}K_{f_k}^{2/3}(p)\rho_0^{4/3}(p)d\mathcal{L}^3(p)=%&=&\frac{3}{4}\int_{2\log a}^{2\log b}\int_{-2\pi/3}^{2\pi/3}\int_{-\pi/2}^{\pi/2}e^{2\xi}\left(\frac{1+\tan^2\psi}{k^2+\tan^2\psi}\right)^{2/3}\times\\
%&&\times\frac{e^{-2\xi}\cos^{-2/3}\psi}{\left(4\pi\log(b/a)\right)^{4/3}}d\psi d\eta d\xi\\
%&=&\left(2^5\pi\log(b/a)\right)^{-1/3}\int_{-\pi/2}^{\pi/2}\frac{1+\tan^2\psi}{(k^2+\tan^2\psi)^{2/3}}d\psi\\
%&=&
{\rm Mod}(f_k(\Sigma_0)).
\end{eqnarray*}

\subsubsection*{\bf{Step 3: Conclusion of Proof}}
To conclude the proof of Theorem \ref{thm:main} we must define the larger family $\Sigma$ and show that $
 f_k(\Sigma_0)\subseteq f(\Sigma)
 $ and  that $\rho_0$ is admissible for $\Sigma$. A natural choice is the one of the family of regular surfaces that  join the two boundaries of the spherical ring. But this is a rather large family and it contains surfaces which might not at all look like  Heisenberg cones, even locally. To that end, we examine the particular properties of Heisenberg cones. Let $\calC$ be such a cone inside $S_{a,b}$ and let $\widetilde{\calC}$ be $\Phi^{-1}(\calC)$. Then the following hold.
 \begin{enumerate}
 \item The surface $\widetilde{\calC}$ admits a parametrisation
 $$
 {\widetilde \sigma}(\xi,\eta)=\left(\Xi(\xi,\eta),\Psi(\xi,\eta),{\rm H}(\xi,\eta)\right)=(\xi,c,\eta), \quad (\xi,\eta)\in(2\log a, 2\log b)\times(-2\pi/3,2\pi/3),
 $$
 that is, $\psi=\psi(\xi,\eta)=c$.
 \item The induced 1-form $\omega_\calC$ is given in logarithmic coordinates by
 $$
 \omega_{\widetilde{\calC}}=\Phi^{*}(\omega_\calC)=-e^\xi\cos c(\tan cd\xi+3d\eta)
 $$
 and therefore $\ker\;\omega_{\calC}$ is generated by $ \partial_\xi-(\tan c /3)\partial_\eta$ which may be identified to the restriction of  %\left\langle
$\Phi_*(W_{\Xi,{\rm H}})$ on $\calC$:
$
{\widetilde \sigma}_*(\partial_\xi-(\tan c /3)\partial_\eta)=W_{\Xi,{\rm H}}.
$
 Note that this also implies that 
the restriction of the vector field $\Phi_*(\partial/\partial\Psi)$ on $\calC$ is horizontally  transverse to the horizontal space of $\calC$. 
 \end{enumerate}
 One may now check that a surface $\E$ which is such that $\widetilde{\E}=\Phi^{-1}(\E)$ admits a local parametrisation
 $$
{\widetilde \sigma}(\xi,\eta)=\left(\Xi(\xi,\eta),\Psi(\xi,\eta),{\rm H}(\xi,\eta)\right), \quad (\xi,\eta)\in U=(2\log a,2\log b)\times (-2\pi/3,2\pi/3),
$$
 has the property that $\ker\;\omega_\E$ is generated by a vector field which can be identified to $\Phi_*(V)$ where
\begin{equation}\label{eq:V}
V=\frac{\partial(\Xi,{\rm H})}{\partial(\xi,\eta)}W_{\Xi,{\rm H}}+3\left(\frac{\partial(\Psi,{\rm H})}{\partial(\xi,\eta)}-\frac{\tan\Psi}{3}\frac{\partial(\Xi,\Psi)}{\partial(\xi,\eta)}\frac{\partial}{\partial\Psi}\right).
\end{equation}
Note that $\partial(\Xi,{\rm H})/\partial(\xi,\eta)\neq 0$ allow us to reparametrise so that $\widetilde C$ is locally of the form $\Psi=\Psi(\xi,\eta)$.
 
 \medskip
 
The previous observations drive us to impose the following additional condition (C) for $\Sigma$:

\medskip

(C) {\it  Every $\calS\in\Sigma$ is such that $\widetilde{\E}$ can be locally parametrised away from its characteristic locus as}
$$
\Psi=\Psi(\xi,\eta),\quad (\xi,\eta)\in U=(2\log a,2\log b)\times (-2\pi/3,2\pi/3).
$$

\medskip

 In the first place, if $\Sigma$ satisfies (C) then for any arbitrary contact quasiconformal map  $f\in\mathcal F$ we have
 $
 f_k(\Sigma_0)\subseteq f(\Sigma).
 $
 To see this, suppose  that we have a Heisenberg cone $\calC\in f_k(\Sigma_0)$ where $\widetilde{\calC}$ is given by $\psi=c$ and an arbitrary contact qc map $f\in\mathcal{F}$. The surface  $\calC'=f^{-1}(\calC)$ admits a local parametrisation
$$
(\xi,\eta)\mapsto\left(\Xi(\xi,c,\eta),\Psi(\xi,c,\eta),{\rm H}(\xi,c,\eta)\right)
$$ 
where  $(\Xi,\Psi,{\rm H})={\tilde f}^{-1}$ and thus its horizontal tangent space is generated by $V$ as in (\ref{eq:V}). Suppose that $\partial(\Xi,{\rm H})=0$ in the neighborhood of some point.  From the contact conditions for $f^{-1}$ we find
 $$
 \tan\Psi\Xi_\xi+3{\rm H}_\xi=\frac{\tan c}{3}(\tan\Psi\Xi_\eta+3{\rm H}_\eta) 
 $$
hence
$$
\frac{\tan\Psi}{3}=\frac{\tan c{\rm H}_\eta-3{\rm H}_\xi}{3\Xi_\xi-\tan c\Xi_\eta}.
$$
The denominator is different from zero; else, we would have a quasilinear PDE $3\Xi_\xi-\tan c\Xi_\eta=0$ whose solution is $\Xi=$const. From the contact condition it also follows that ${\rm H}=$const. and we cannot have a surface.

Now we calculate:
\begin{eqnarray*}
 \frac{\partial(\Psi,{\rm H})}{\partial(\xi,\eta)}-\frac{\tan\Psi}{3}\frac{\partial(\Xi,\Psi)}{\partial(\xi,\eta)}&=&\Psi_\xi{\rm H}_\eta-\Psi_\eta{\rm H}_\xi-\frac{\tan\Psi}{3}\cdot(\Xi_\xi\Psi_\eta-\Psi_\xi\Xi_\eta)\\
&=&\frac{(\Psi_\xi{\rm H}_\eta-\Psi_\eta{\rm H}_\xi)(3\Xi_\xi-\tan c\Xi_\eta)-(\Xi_\xi\Psi_\eta-\Psi_\xi\Xi_\eta)(\tan c{\rm H}_\eta-3{\rm H}_\xi)}{3\Xi_\xi-\tan c\Xi_\eta}\\
&=&\frac{3\Psi_\xi-\tan c\Psi_\eta}{3\Xi_\xi-\tan c\Xi_\eta}\cdot\frac{\partial(\Xi,{\rm H})}{\partial(\xi,\eta)}=0.
\end{eqnarray*}
Thus $V=0$ which is a contradiction since we are away from the characteristic locus. Therefore $\calC'\in\Sigma$ which  implies that $\calC=f(\calC')\in f(\Sigma)$.

\medskip

We finally show that $\rho_0\in{\rm Adm}(\Sigma)$. For this, let $\E\in\Sigma$ and we may suppose that it is parametrised  by the single surface patch
$$
\sigma(\xi,\eta)=\left(z(\xi,\eta),t(\xi,\eta)\right)=\left(i\cos^{1/2}(\psi(\xi,\eta))
e^{\frac{\xi+i(\psi(\xi,\eta)-3\eta)}{2}},-e^\xi\sin\psi(\xi,\eta)\right).
$$
This assumption is allowed by our condition (C) and the fact that the horizontal area integral does not depend on the choice of coordinates, see Section \ref{sec:horarea}. We have
\begin{eqnarray*}
 &&
 z_\xi=\frac{z}{2}\cdot \left((1-\tan\psi\cdot\psi_\xi)+i\psi_\xi\right),\quad
  z_\eta=\frac{z}{2}\cdot \left(-\tan\psi\cdot\psi_\eta)+i(\psi_\eta-3)\right),\\
  &&
  t_\xi=-e^{\xi}\left(\sin\psi+\cos\psi\cdot\psi_\xi\right),\quad
  t_\eta=-e^\xi\cos\psi\cdot\psi_\eta.
\end{eqnarray*}
Therefore
\begin{eqnarray*}
 \sigma_\xi=z_\xi Z+\overline{z_\xi}\overline{Z}-a T,\\
 \sigma_\eta=z_\eta Z+\overline{z_\eta}\overline{Z}-b  T,
\end{eqnarray*}
where
$
a=e^\xi\sin\psi$, $b=3e^\xi\cos\psi.
$
Thus
\begin{eqnarray*}
 N_\sigma^h=(\sigma_\xi\wedge^\fH \sigma_\eta)^h=i\left((bz_\xi-az_\eta)Z-(b\overline{z_\xi}-a\overline{z_\eta})\overline{Z}\right).
\end{eqnarray*}
Straightforward calculations now yield
\begin{eqnarray*}
 \|N_\sigma^h\|&=&|bz_\xi-az_\eta|\\
 &=&\frac{3}{2}e^{3\xi/2}\cos^{1/2}\psi\left(1+\left(\psi_\xi-\frac{\tan\psi\cdot\psi_\eta}{3}\right)^2\right)^{1/2}.
\end{eqnarray*}
Hence
\begin{eqnarray*}
 \iint_\E\rho_0d\E^h&=&\iint_U\rho_0 \|N_\sigma^h(\xi,\eta)\|d\xi d\eta\\
 &=&\frac{1}{4\pi\log(b/a)}\iint_Ue^{-3\xi/2}\cos^{-1/2}\psi\cdot\frac{3}{2}e^{3\xi/2}\cos^{1/2}\psi\times\\
&&\times\left(1+\left(\psi_\xi-\frac{\tan\psi\cdot\psi_\eta}{3}\right)^2\right)^{1/2}d\xi d\eta\\
 &=&\frac{3}{8\pi\log(b/a)}\int_{2\log a}^{2\log b}\int_{-2\pi/3}^{2\pi/3}\left(1+\left(\psi_\xi-\frac{\tan\psi\cdot\psi_\eta}{3}\right)^2\right)^{1/2}d\xi d\eta\\
 &\ge &\frac{3}{8\pi\log(b/a)}\int_{2\log a}^{2\log b}\int_{-2\pi/3}^{2\pi/3} d\xi d\eta\\
 &=&1.
\end{eqnarray*}
The proof of Theorem \ref{thm:main} concludes here. \hfill$\Box$

\end{document}